\documentclass[a4paper,11pt]{amsart}

\usepackage[a4paper]{geometry}
\usepackage{amsmath,amsthm}
\usepackage{amssymb,amsfonts}
\usepackage{hyperref}
\usepackage{tikz}
\usepackage{enumerate}
\usepackage{graphicx}
\usepackage{calligra}
\usepackage{stmaryrd}
\DeclareSymbolFont{bbold}{U}{bbold}{m}{n}
\DeclareSymbolFontAlphabet{\mathbbm}{bbold}

\title{Trace spaces of counterexamples to Naimark's problem}

\theoremstyle{plain}
	\newtheorem{theorem*}{Theorem}
	\newtheorem{theorem}{Theorem}[section]
	\newtheorem{thrm}{Theorem}[section]
	
	\newtheorem{proposition}[theorem]{Proposition}
	\newtheorem{lemma}[theorem]{Lemma}
	
	\newtheorem{corollary}[theorem]{Corollary}
	\newtheorem{claim}{Claim}[theorem]

\theoremstyle{definition}
	\newtheorem*{NP*}{Naimark's problem}
	\newtheorem*{cpou*}{Complemented Partitions of Unity}
	\newtheorem*{dia*}{The diamond principle ($\diamondsuit$)}
	\newtheorem{definition}[theorem]{Definition}

	\newtheorem{question}[theorem]{Question}
	\newtheorem*{acknow}{Acknowledgements}
\theoremstyle{remark}
	\newtheorem{remark}[theorem]{Remark}

\newcommand{\C}{\mathbb{C}}

\newcommand{\N}{\mathbb{N}}

\newcommand{\A}{A}
\newcommand{\cstar}{$\mathrm{C}^\ast$}
\newcommand{\set}[1]{\{#1\}}					% set
\newcounter{my_enumerate_counter}
\newcommand{\pushcounter}{\setcounter{my_enumerate_counter}{\value{enumi}}}
\newcommand{\popcounter}{\setcounter{enumi}{\value{my_enumerate_counter}}}

\renewcommand{\phi}{\varphi} %%%%% I bother to write \varphi

\usepackage[mode=multiuser,status=draft,lang=english]{fixme}
\fxsetup{theme=colorsig}
\hypersetup{colorlinks=false, pdfborder={0 0 0}}

\usepackage{titlesec}
\titlelabel{\thetitle. \;}
\titleformat*{\section}{\centering\large\scshape}

\FXRegisterAuthor{M}{aM}{A.Vaccaro}

\author{Andrea Vaccaro}
\date{}

\begin{document}
\maketitle
\begin{abstract}
A counterexample to Naimark's problem is a \cstar-algebra
that is not isomorphic to the algebra of compact operators on some Hilbert space,
yet still has only one irreducible representation up to unitary equivalence.
It is well-known that such algebras must be nonseparable,
and in 2004 Akemann and Weaver used the diamond principle
(a set theoretic axiom independent from \textsf{ZFC}) to give the first counterexamples.
For any such counterexample $A$, the unitary group $\mathcal{U}(A)$
acts transitively on the pure states,
which are the extreme points of the state space $S(A)$.
It is conceivable that this implies (as happens for finite-dimensional simplexes)
that the action of $\mathcal{U}(A)$ on $S(A)$ has at
most one fixed point, i.e. $A$ has at most one trace.
We give a strong negative answer here assuming diamond. In particular, 
we adapt the Akemann-Weaver construction to show that the trace space
of a counterexample to Naimark's problem can be affinely homeomorphic to
any metrizable Choquet simplex, and can also be nonseparable.
\end{abstract}

\keywords

\section{Introduction} \label{sctn:1}
In 1948, in \cite{Naimark1}, Naimark observed that the algebra of compact
operators $K(H)$ has a unique irreducible representation up to unitary equivalence
(the identity representation).
A few years later, in \cite{Naimark2}, he
asked whether this property characterizes $K(H)$ up to isomorphism.
The following question is known as \emph{Naimark's problem}.
\begin{NP*}
Let $A$ be a \cstar-algebra with
only one irreducible representation up to unitary equivalence. Is $A \cong K(H)$
for some Hilbert space $H$?
\end{NP*}

In the years immediately after Naimark posed his question, it was shown that the problem
has a positive answer for the classes of type I and separable \cstar-algebras
(see \cite[Theorem 7.3]{kaplansky} and \cite[Theorem 4]{rosenberg}). More recently,
a positive answer to the problem was proved also for certain graph \cstar-algebras (see \cite{graph}).

Therefore, a \emph{counterexample to Naimark's problem} would have to be a simple, non-type I,
nonseparable \cstar-algebra with a unique equivalence class of irreducible
representations. In 2004 Akemann and Weaver
were able to produce, with the extra set-theoretic axiom known as $\diamondsuit$,
a unital example of such algebras (see \cite{AW}). We remark that it is still not known
whether a positive answer to Naimark's problem is consistent with \textsf{ZFC}.

The general motivation of our inquiry is to investigate further the structural
and \cstar-algebraic properties that counterexamples
to Naimark's problem can (or cannot) satisfy,
other than those mentioned above. In this paper we focus on the study
of trace spaces, led by the following general observation regarding group actions on compact convex sets,
which initially seemed to suggest some kind of limitation on the size
of tracial simplexes
of counterexamples to Naimark's problem. Before going any further, we remark that
the original construction of the counterexamples given by Akemann and Weaver does not
provide any immediate information on the tracial simplex of these algebras, except maybe
that there both counterexamples with no traces and counterexamples whose trace space
is non-empty (see proposition \ref{prop:0} and the paragraph after corollary \ref{corollary:0}).

Let $K$ be a compact convex set and $G$ a group of affine homeomorphisms of $K$
and consider the action
\begin{align*}
\Theta: G \times K &\to K \\
(g, x) &\mapsto g(x)
\end{align*}
Assume moreover that the action is transitive when restricted to the set of extreme points.
It is conceivable that the set of the points in $K$ fixed by the action has size no bigger than
one (as happens if $K$ is a finite-dimensional simplex)\footnote{In this case, if there are at least
two points fixed by $\Theta$, we can find a
point $y = \sum_{k \le n} \lambda_k x_k$ such that $g(y)=y$ for all $g \in G$, and
$\lambda_i \not= \lambda_j$ for some $i \not= j$,
where $\set{x_1,\dots, x_n}$ are affinely independent extremal points of $K$. However,
for any $g \in G$ such that
$g(x_i) = x_j$, we get $g(y) \not = y$.}. This relates to our context as follows.
In a unital counterexample to Naimark's problem $A$ there is a unique irreducible
representation modulo unitary equivalence. This implies, by \cite[Theorem 5.1.4]{murphy} and
an application of Kadison transitivity theorem (\cite[Theorem 5.2.2]{murphy}),
that the action of the unitary group on the state space of $A$
\begin{align*}
\Theta_A: \mathcal{U}(A) \times S(A) &\to S(A) \\
(u, \phi) &\mapsto \phi \circ \text{Ad}(u)
\end{align*}
is transitive on the pure states of $A$, namely the extreme points of $S(A)$.
Moreover, since the traces
are fixed by this action, according to the previous observation it may seem plausible
that a counterexample to Naimark's problem could have at most one trace.
This is not the case for general affine actions $\Theta$,
in fact there is no strict bound on the number of fixed points even for a separable $K$.
For instance, let $A$ be a separable simple unital \cstar-algebra and let $\text{AInn}(A)$
be the group of \emph{asymptotically inner} automorphisms of $A$, i.e, the group of all
$\alpha \in \text{Aut}(A)$ such that there exists a continuous path
$(u_t)_{t \in [0, \infty)} \subseteq \mathcal{U}(A)$ such that
$\alpha(a) = \lim_{t \to \infty} \text{Ad}(u_t) (a)$ for all $a \in A$.
Then, by the Kishimoto-Ozawa-Sakai theorem on the transitivity
of the action of automorphisms on the pure state space of a separable simple unital \cstar-algebra
in \cite{KOS}, the action
\begin{align*}
\Xi_A: \text{AInn}(A) \times S(A) &\to S(A) \\
(\alpha, \phi) &\mapsto \phi \circ \alpha
\end{align*}
is transitive on the extreme points of $S(A)$. On the other hand, since traces are fixed by inner
automorphisms, by continuity they are also fixed by the elements of $\text{AInn}(A)$.
Since every metrizable
Choquet simplex occurs as the trace space of some separable simple
unital \cstar-algebra (see \cite{blackAF}), we infer that the set
of fixed points in $\Xi_A$ can be considerably large.
The same is true for the unitary action $\Theta_A$ on the state space
of a counterexample to Naimark's problem, as is shown in the main result
of this paper.

\begin{theorem*} \label{thrm:1}
Assume $\diamondsuit$. Then the following holds:
\begin{enumerate}
\item For every metrizable Choquet simplex $X$ there is a counterexample to Naimark's problem whose trace
space is affinely homeomorphic to $X$.
\item There is a counterexample to Naimark's problem whose trace space is nonseparable.
\end{enumerate}
\end{theorem*}
In fact, we obtain the following stronger result.
\begin{theorem*} \label{thrm:2}
Assume $\diamondsuit$. For every metrizable Choquet simplex $X$ and
$1 \le n \le \aleph_0$ there is a \cstar-algebra $A$ such that
\begin{enumerate}
\item $A$ is simple, unital, nuclear and of density character $\aleph_1$,
\item $A$ is not isomorphic to its opposite algebra,
\item $A$ has exactly $n$ equivalence classes of pure states,
\item all automorphisms of $A$ are inner,
\item either of the following conditions can be obtained:
\begin{enumerate}
\item $T(A)$ is affinely homeomorphic to $X$. \label{clause1}
\item $T(A)$ is nonseparable. \label{clause2}
\end{enumerate}
\end{enumerate}
\end{theorem*}

In \cite{glimm} Glimm
shows that every separable \cstar-algebra which is not type I has
uncountably many inequivalent irreducible representations.
We remark how theorem \ref{thrm:2} (in particular its third clause)
pushes even further the consistency of the failure of Glimm's dichotomy
in the nonseparable setting, already obtained in \cite{AW} and \cite{farahilan}
(see also section 8.2 of \cite{faICM}).

The starting point for the proof of theorem \ref{thrm:2} is the techniques developed in
\cite{AW} and \cite{farahilan},
which both rely on an application
of the Kishimoto-Ozawa-Sakai theorem in \cite{KOS}.
As we shall see in the next section, the main effort to prove theorem \ref{thrm:2}
will be to refine the results in \cite{KOS}
in order to have a better control on the trace space
of the crossed product obtained from the automorphism provided by the Kishimoto-Ozawa-Sakai theorem
(see theorem \ref{thrm:a} in section \ref{sctn:2}). We remark that a
key tool in the proofs of the main technical
lemmas of the paper (lemmas \ref{lemma:2+} and \ref{lemma:1+}) are the so-called \emph{complemented
partitions of unity} recently defined in \cite{cpou}. We shall recall their definition in the next section.

The paper is organized as follows. We fix some definitions and notations in
section \ref{sctn:1.5}.
In section \ref{sctn:2} we show how the study
of the trace space of a counterexample to Naimark's problem is reduced to
a refinement the Kishimoto-Ozawa-Sakai theorem. In section
\ref{sctn:3} we prove two main lemmas which will be necessary
in section \ref{sctn:4}, where our variant of the Kishimoto-Ozawa-Sakai theorem is proved.
Finally section \ref{sctn:5} is devoted to conclusions and final observations. We remark
that no additional set-theoretic axiom is needed in sections \ref{sctn:3} and \ref{sctn:4}.

\section{Preliminaries} \label{sctn:1.5}
First we fix some notations and definitions.
\subsection{\cstar-algebras}
If $A$ is a \cstar-algebra, $A_{sa}$ is the set of its
self-adjoint elements, $A_+$ the set of its positive elements and $A^1$ the set
of its norm one elements. If $A$ is unital, $\mathcal{U}(A)$ is the set of all unitaries in $A$.
Denote by $S(A)$ the state space, by $P(A)$ the pure state space,
by $T(A)$ the trace space, and by $\partial T(A)$ the set of extremal traces of $A$, all endowed
with the weak* topology. Through this paper by \emph{trace} of a given \cstar-algebra $\A$,
we always mean a tracial state of $\A$, namely a norm-one, positive, linear
functional $\tau$ of $\A$ such that $\tau(ab) = \tau(ba)$ for all $a,b \in \A$.
We denote by the symbol $\sim$ the unitary equivalence between states.
Given any $\phi \in S(A)$, $(\pi_\phi, H_\phi, \xi_\phi)$ is the GNS representation
associated to $\phi$.
If $\tau \in T(A)$, for $A$ a simple \cstar-algebra,
we denote by $\lVert \phantom{a} \rVert_{2, \tau}$ the $\ell_2$-norm induced by $\tau$ on $A$
(the subscript $\tau$ will be suppressed when there is no risk of confusion).
Given two vectors $\xi$ and $\eta$ in a normed vector space, $\xi \approx_\epsilon \eta$
means $\lVert \xi - \eta \rVert < \epsilon$. For functions $\phi$ and $\psi$ on a normed vector space, given a finite subset $G$ of the vector space and $\delta > 0$,
$\phi \approx_{G, \delta} \psi$ means $\lVert \phi(\xi) - \psi(\xi) \rVert < \delta$ for all $\xi \in G$.
Denote by $\text{Aut}(A)$ the set of all automorphisms of $A$.
Given a unital \cstar-algebra $A$ and $u \in \mathcal{U}(A)$, the \emph{inner automorphism}
induced by $u$ on $A$ is $\text{Ad}(u)$ and it sends $a$ to $uau^*$. An automorphism
$\alpha$ is \emph{outer} if it is not induced by a unitary, and we denote by $\text{Out}(A)$
the set of all outer automorphisms. An automorphism
$\alpha \in \text{Aut}(A)$ is \emph{asymptotically inner}
if there exists a continuous path of unitaries $(u_t)_{t \in [0, \infty)}$ in $A$
such that $\alpha(a) = \lim_{t \to \infty} \text{Ad}(u_t) (a) $ for all $a \in A$.
Given $\alpha \in \text{Aut}(A)$ and $\tau \in T(A)$, the trace $\tau$ is $\alpha$\emph{-invariant}
if $\tau(\alpha(a)) = \tau(a)$ for all $a \in A$.
Suppose $A$ is simple and
unital, let $\alpha \in \text{Aut}(A)$, $\tau \in T(A)$ and suppose furthermore that $\tau$
is $\alpha$-invariant.
Then $\alpha$ can be canonically extended to an automorphism of $\pi_\tau(A)''$,
(see \cite[Section 2]{bedos}). The
automorphism
$\alpha$ is \emph{$\tau$-weakly inner} (\emph{$\tau$-strongly outer}) if its canonical extension to $\pi_\tau(A)''$
is inner (outer).

\subsection{Set theory}
$\aleph_1$ is the smallest uncountable cardinal, the well-ordered set
of all countable ordinals. A \emph{club} in $\aleph_1$ is an unbounded subset $C \subseteq \aleph_1$ such
that for every increasing sequence 
$\set{\beta_n}_{n \in \N} \subseteq C$ the supremum $\sup_{n \in \N} \set{\beta_n}$ belongs to $C$. A subset of $\aleph_1$
is \emph{stationary} if it meets every club.
An increasing transfinite $\aleph_1$-sequence of \cstar-algebras $\set{A_\beta}_{\beta
< \aleph_1}$ is \emph{continuous} if $A_\gamma = \overline{ \cup_{\beta < \gamma} A_\beta}$
for every limit ordinal $\gamma < \aleph_1$.

The following is Jensen's original formulation of $\diamondsuit$.
\begin{dia*}
There exists an $\aleph_1$-sequence of sets $\set{A_\beta}_{\beta < \aleph_1}$ such that
\begin{enumerate}
\item $A_\beta \subseteq \beta$ for every $\beta < \aleph_1$,
\item for every
$A \subseteq \aleph_1$ the set $\set{\beta < \aleph_1 : A \cap \beta = A_\beta}$ is stationary.
\end{enumerate}
\end{dia*}
The principle $\diamondsuit$ is known to be true in the G\"odel
constructible universe (\cite[Theorem 13.21]{jech}) and it implies the continuum hypothesis,
thus it is independent from
the Zermelo-Fraenkel axiomatization of set theory plus the Axiom of Choice (\textsf{ZFC}).

\subsection{Tracial Ultrapowers and Complemented Partitions of Unity}
Given a \cstar-algebra $A$, for every $\emptyset \not = T \subseteq T(A)$ we define the seminorm
\[
\lVert a \rVert_{2, T} = \sup_{\tau \in T} \tau(a^*a)^{1/2},
\]
which is always a norm in case $A$ is simple.
We denote $\lVert \cdot \rVert_{2, T(A)}$ by $\lVert \cdot \rVert_{2, u}$.
For a free ultrafilter $\mathcal{U}$ on $\N$, the \emph{ultrapower} of $A$ is
\[
A_\mathcal{U} = \ell_\infty(A) /c_{\mathcal{U}},
\]
where
\[
c_{\mathcal{U}} = \set{ \vec{a} \in \ell_\infty(A) : \lim_{n \to \mathcal{U}}\lVert a_n \rVert= 0 }.
\]
The \emph{uniform tracial ultrapower} of $A$ is
\[
A^{\mathcal{U}} = A_{\mathcal{U}} / J_A,
\]
where $J_A$ is the \emph{trace-kernel ideal}:
\[
J_A = \set{ \vec{a} \in A_{\mathcal{U}} : \lim_{n \to \mathcal{U}} \lVert a_n \rVert_{2,u} = 0}.
\]
We freely identify elements of $\ell_\infty(A)$ with their classes in $A_\mathcal{U}$ or
in $A^{\mathcal{U}}$. The sets of all elements in $A_\mathcal{U}$ and in $A^\mathcal{U}$ which
commute with all elements in (the canonical copy of) $A$ are respectively denoted by $A_\mathcal{U} \cap A'$ and $A^\mathcal{U} \cap A'$.

Any sequence of traces $(\tau_n)_{n \in \N}$ on $A$ induces a trace $\tau$ on $A_{\mathcal{U}}$
\[
\tau(\vec{a}) = \lim_{n \to \mathcal{U}} \tau_n(a_n).
\]
The set of all such limit traces on $A_{\mathcal{U}}$ is $T_{\mathcal{U}}(A)$. With
this notation we have that
\[
J_A = \set{ \vec{a} \in A_{\mathcal{U}} : \lVert \vec{a} \rVert_{2, T_{\mathcal{U}}(A)} = 0}.
\]
We introduce now the notion of complemented partitions of unity of a \cstar-algebra, originally
defined in \cite[Definition G - Definition 3.1]{cpou} (see also \cite[Definition 1.13]{cpou2}).
\begin{definition}[Complemented Partitions of Unity]
Let $A$ be a separable \cstar-algebra such that $T(A)$ is non-empty and compact. We say
that $A$ has \emph{complemented partitions of unity} (abbreviated \emph{CPoU}) if for every
$F \Subset A$, every $\epsilon >0$, every $a_1, \dots a_k \in A_+$ and every scalar $\delta > 0$
such that
\[
\sup_{\tau \in T(A)} \min_{i \le k} \tau(a_k) < \delta,
\]
there exist elements $e_1, \dots, e_k \in A^1_{+}$ such that
\begin{enumerate}
\item $\lVert [e_i, b] \rVert < \epsilon$ for all $b \in F$ $i \le k$,
\item $\sum_{i \le k} e_i = 1$,
\item $\lVert e_i^2 - e_i \rVert_{2,u} < \epsilon$,
\item $\lVert e_i e_j \rVert_{2, u} < \epsilon$,
\item $\tau(a_i e_i) < \delta \tau(e_i) + \epsilon$ for all $i \le k$ and $\tau \in T(A)$.
\end{enumerate}
\end{definition}

\begin{remark} \label{naim:remark}
Let $A$ be a simple \cstar-algebra such that $M_2(\C)$ embeds into $A^{\mathcal{U}} \cap A'$.
It follows that $M_n(\C)$ embeds into $A^{\mathcal{U}} \cap A'$ for every $n \in \N$. This can be
shown as follows.
Let $B$ be a copy of $M_2(\C)$ in $A^{\mathcal{U}} \cap A'$. By standard
arguments it can be shown that there is a copy of $M_2(\C)$ also in
$(A \cup B)' \cap A^{\mathcal{U}}$. This can be done using model theory
via the countable saturation of $A^{\mathcal{U}}$ or, by more standard means,
using Kirchberg's $\epsilon$-test, \cite[Lemma A.1]{kir}. Iterating this operation countably many times,
it is possible to define an embedding $\phi$
of the \textsf{CAR} algebra $M_{2^\infty}$ into $A^{\mathcal{U}} \cap A'$.
Let $\tau$ be the unique trace of $M_{2^\infty}$. Since the unit ball of $A^{\mathcal{U}} \cap A'$
is $\lVert \cdot \rVert_{2,T_{\mathcal{U}}(A)}$-complete (\cite[Lemma 1.6]{cpou})
and the embedding $\phi$
is $(\lVert \cdot \rVert_{2, \tau} - \lVert \cdot \rVert_{2,T_{\mathcal{U}}(A)})$-contractive (by uniqueness 
of the trace on $M_{2^\infty}$),
by Kaplansky's density theorem $\phi$ extends by continuity to an embedding of the hyperfinite
$\text{II}_1$ factor $\mathcal{R}$ into $A^{\mathcal{U}} \cap A'$. Since $M_n(\C)$ embeds into
$\mathcal{R}$ for every $n \in \N$, it follows that it also embeds into $A^{\mathcal{U}} \cap A'$.
\end{remark}

The matrix algebra $M_2(\C)$ is the universal \cstar-algebra generated by two unitaries $u,v$
such that $uv = - vu$ (\cite[Lemma 1.3]{farahCCR}). We have therefore the following proposition.
\begin{proposition} \label{prop:comm}
Let $A$ be a separable, simple, unital \cstar-algebra. The following are equivalent
\begin{enumerate}
\item $M_2(\C)$ embeds into $A^{\mathcal{U}} \cap A'$,
\item \label{ni2} for all finite $F \Subset A$ and $\epsilon > 0$ there are $u,v \in A$ such that
\begin{enumerate}
\item $\max \set{ \lVert u \lVert, \rVert v \rVert} \le 1$,
\item $\max \set{\lVert uu^*-1 \rVert_{2,u}, \lVert vv^* - 1 \rVert_{2,u},\lVert u^*u-1 \rVert_{2,u},
\lVert v^*v - 1 \rVert_{2,u}, \lVert uv + vu \rVert_{2,u}} < \epsilon$,
\item $\max_{b \in F} \set{\lVert [u,b] \rVert_{2,u} , \lVert [v,b] \rVert_{2,u}} < \epsilon$.\end{enumerate}
\end{enumerate}
\end{proposition}
Let $S \subseteq A$ be a set of generators of $A$.
The equivalence above holds also if clause \ref{ni2} is checked for all $F \Subset S$,
rather than $F \Subset A$.

Combining remark \ref{naim:remark}
with \cite[Lemma 3.6]{cpou} and \cite[Lemma 3.7]{cpou} we obtain
the following theorem.
\begin{theorem} \label{cpou}
Let $A$ be a nuclear, separable, unital \cstar-algebra such that $A^{\mathcal{U}} \cap A'$
contains a copy of $M_2(\C)$, for some free ultradilter $\mathcal{U}$. Then $A$ has CPoU.
\end{theorem}
We remark that it is not known whether the nuclearity assumption is required to prove the previous result.

\section{The trace space of a counterexample to Naimark's problem} \label{sctn:2}
Akemann-Weaver's proof in \cite{AW} is the starting point of our analysis.
Their results were refined in \cite{farahilan} to produce,
given $1 \le n \le \aleph_0$, a non-type I \cstar-algebra
$A$ not isomorphic to its opposite, with exactly $n$ equivalence classes of
irreducible representations, and with no outer automorphisms.
For the reader's convenience we quickly recall the construction here.
All omitted details can be found in \cite{farahilan}, where a continuous model-theoretic equivalent
version of $\diamondsuit$, more suitable for working with \cstar-algebras, is introduced.

The algebra $A$ is obtained as an inductive limit of
an increasing continuous $\aleph_1$-sequence of separable simple unital \cstar-algebras
\[
A_0 \subseteq A_1 \subseteq \dots \subseteq A_\beta \subseteq \dots \subseteq A = 
\bigcup_{\beta < \aleph_1} A_\beta
\]
where each inclusion is unital.
The crucial part of the construction is the successor step, where the following
improvement of the main result of \cite{KOS} is used.
\begin{theorem}[\cite{AW}] \label{KOS2}
Let $A$ be a separable simple unital \cstar-algebra, and let $\set{\phi_h}_{h \in \N}$ and
$\set{\psi_h}_{h \in \N}$ be two sequences of pure states of $A$ such that the $\phi_h$'s are mutually
inequivalent, and similarly the $\psi_h$'s. Then there is an asympotically inner automorphism
$\alpha$ such that $\phi_h \sim \psi_h \circ \alpha$ for all $h \in \N$.
\end{theorem}
Theorem \ref{KOS2} is applied in the proof of the following lemma.
\begin{lemma}[{\cite[Lemma 2.3]{farahilan}}] \label{lemma}
Let $A$ be a separable simple unital \cstar-algebra. Suppose $\mathcal{X}$ and $\mathcal{Y}$
are disjoint countable sets of inequivalent pure states of $A$ and let $E$ be an equivalence relation
on $\mathcal{Y}$. Then there exists a separable simple unital \cstar-algebra $B$ such that
\begin{enumerate}
\item $B$ unitally contains $A$,
\item every $\psi \in \mathcal{X}$ has multiple extensions to $B$,
\item every $\phi \in \mathcal{Y}$ extends uniquely to a pure state $\tilde{\phi}$ of $B$,
\item given $\phi_0, \phi_1 \in \mathcal{Y}$, then $\phi_0 E \phi_1$ if and only if
$\tilde{\phi}_0 \sim \tilde{\phi}_1$.
\end{enumerate}
\end{lemma}
The algebra $B$ is $A \rtimes_\alpha \mathbb{Z}$, where $\alpha \in \text{Aut}(A)$
is provided by theorem \ref{KOS2} for two sequences of inequivalent pure states which
depend on $\mathcal{X}$, $\mathcal{Y}$ and $E$.
Back to the construction in \cite{farahilan}, given
$A_\beta$, $A_{\beta +1}= A_\beta \rtimes_\alpha \mathbb{Z}$
is obtained by an application of lemma \ref{lemma},
where $\mathcal{X}, \mathcal{Y}$ and $E$ are chosen according to $\diamondsuit$.

We warm up proving the following simple fact.
\begin{proposition} \label{prop:0}
Let $\set{A_\beta}_{\beta < \aleph_1}$ be an increasing continuous $\aleph_1$-sequence of
unital \cstar-algebras such that
$A_{\beta+1} = A_\beta \rtimes_{\alpha,r} G_\beta$
for all $\beta < \aleph_1$, $G_\beta$ being a discrete group.
Let $A$ be the inductive limit of the sequence.
Suppose furthermore
that every $\tau \in T(A_\beta)$ is $\alpha_g$-invariant
for all $g \in G_\beta$.
Then for each $\beta < \aleph_1$ there is an embedding\footnote{A continuous map which
is a homeomorphism with the image.} $e_\beta$ of $T(A_\beta)$ into $T(A)$.
\end{proposition}
\begin{proof}
Let $B$ be any unital tracial \cstar-algebra, $\tau \in T(B)$, and $\alpha$ a homomorphism of 
a discrete group $G$ (whose identity is $e$) into $\text{Aut}(B)$
such that $\tau$ is $\alpha_g$-invariant for all $g \in G$.
Consider the reduced crossed product $B \rtimes_{\alpha,r} G$ and denote by
$u_g$, for $g \in G$, the unitaries corresponding to the elements of the group.
The map defined on any
finite sum $\sum_{g \in G} a_g u_g$ as
\[
\tau'(\sum_{g \in G} a_g u_g) = \tau(a_e)
\]
extends uniquely to a trace of $B\rtimes_{\alpha,r} G$. In fact, $\tau'$ is $\text{Ad}(u)$-invariant
for all $ u \in \mathcal{U} (B)$ since $\tau$ is a trace, and it is $\text{Ad}(u_g)$-invariant for all $g \in G$ since $\tau$ is $\alpha_g$-invariant, hence $\tau'(w a) = \tau'(a w)$ for all $a \in B\rtimes_{\alpha,r} G$ and $w= w_1 \dots w_k$, where 
$w_j \in \mathcal{U}(B) \cup \set{u_g : g \in G}$ for all $j \le k$.
The linear span of the set of products of elements in $\mathcal{U}(B) \cup \set{u_g : g \in G}$
is dense in $B\rtimes_{\alpha,r} G$, therefore $\tau'(ab) = \tau'(ba)$ for all
$a,b \in B\rtimes_{\alpha,r} G$.
Thus, the embedding $e_\beta$ can be constructed by induction iterating
the extension above at successor steps, and taking the unique extension of previous
steps at limit stages.
\end{proof}
In the Akemann-Weaver construction (and in the one from \cite{farahilan} we previously recalled)
there is no restriction, when starting the induction, on the choice of the \cstar-algebra $A_0$,
as long as $A_0$ is separable simple and unital. Since every metrizable
Choquet simplex occurs as the trace space of some separable simple unital \cstar-algebra
(see \cite{blackAF}), and since all traces are invariant for asymptotically
inner automorphisms (as they are pointwise limits of inner automorphisms),
proposition \ref{prop:0} can be applied to the construction we sketched before to have
the following.
\begin{corollary} \label{corollary:0}
Assume $\diamondsuit$.
For every metrizable Choquet simplex $X$ and $1 \le n \le \aleph_0$,
there is a non-type I \cstar-algebra
$A$ not isomorphic to its opposite, with exactly $n$ equivalence classes of
irreducible representations, and with no outer automorphisms, such that $T(A)$
contains a homeomorphic copy of $X$.
\end{corollary}

Proposition \ref{prop:0} implies that the $\aleph_1$-sequence
\[
T(A_0) \xleftarrow{r_{1,0}} T(A_1) \xleftarrow{r_{2,1}} \dots  T(A_\beta) \xleftarrow{r_{\beta+1,\beta}}
\dots \leftarrow T(A)
\]
is a projective system whose bonding maps (the restrictions) are surjective.
Proposition \ref{prop:0} also entails that each restriction has a continuous section.
Theorem \ref{thrm:2} answers affirmatively the questions whether it is possible
to perform the constructions in \cite{AW} and \cite{farahilan} so that the $\aleph_1$-sequence
above `stabilizes',
or so that it is forced to be `strictly increasing'.

Depending on which of the final two clauses of theorem \ref{thrm:2} one wants to obtain, two
different strengthenings of lemma \ref{lemma} are needed.
Clause \ref{clause1} follows
if, when applying lemma \ref{lemma}
to $A = A_\beta$ (hence $B=A_{\beta} \rtimes_\alpha \mathbb{Z}$),
we require in addition that the restriction map $r_{\beta +1, \beta}: T(A_{\beta} \rtimes_\alpha \mathbb{Z}) \to T(A_\beta)$ is a homeomorphism for all $\beta < \aleph_1$. This would in fact
entail that $T(A)$ is affinely homeomorphic to $T(A_0)$. On the other hand, in order
to get clause \ref{clause2}, it is
sufficient to require $r_{\beta + 1, \beta}$ to be not injective for all $\beta < \aleph_1$, as shown
in proposition \ref{prp1}.

Since $\alpha$ is asymptotically inner,
the restriction map $r_{\beta +1, \beta} : T(A_{\beta }\rtimes_\alpha \mathbb{Z}) \to T(A_\beta)$
is a homeomorphism
if and only if all the powers of
$\alpha$
are $\tau$-weakly outer for all $\tau \in \partial T(A)$ (see \cite[Theorem 4.3]{thomsen}).

Thus, all we need to show is the following variant of theorem \ref{KOS2}.
\begin{thrm} \label{thrm:a}
Let $A$ be a separable simple unital \cstar-algebra, and let $\set{\phi_h}_{h \in \N}$ and
$\set{\psi_h}_{h \in \N}$ be two sequences of pure states of $A$ such that the $\phi_h$'s are mutually
inequivalent, and similarly the $\psi_h$'s.
\begin{enumerate}
\item \label{ta:i1} If $A$ is nuclear and there exists a copy of $M_2(\C)$
inside $A^{\mathcal{U}} \cap A'$ for some free ultrafilter $\mathcal{U}$, then
there is an asymptotically inner automorphism $\alpha$ such that $\phi_h \sim \psi_h \circ \alpha$
for all $ h \in \N$, and such that $\alpha^l$ is $\tau$-strongly outer for all $\tau \in \partial T(A)$ and all
$l \in \N$. Moreover, $\alpha$ can be chosen so that
$B:= A \rtimes_{\alpha} \mathbb{Z}$ is such that $M_2(\C)$
embeds into $B^{\mathcal{U}} \cap B'$.
\item \label{ta:i2} Given a countable $T \subseteq \partial T(A)$,
there is an asymptotically inner automorphism $\alpha$ such that $\phi_h \sim \psi_h \circ \alpha$
for all $ h \in \N$ and such that $\alpha$ is $\tau$-weakly inner for all $\tau \in T$.
\end{enumerate}
\end{thrm}

Let's assume theorem \ref{thrm:a} for the rest of this section.
\begin{lemma} \label{lemma:+}
Let $A$ be a separable simple unital \cstar-algebra. Suppose $\mathcal{X}$ and $\mathcal{Y}$
are disjoint countable sets of inequivalent pure states of $A$ and let $E$ be an equivalence relation
on $\mathcal{Y}$. Then there exists a separable simple unital \cstar-algebra $B$ such that
\begin{enumerate}
\item $B$ unitally contains $A$,
\item every $\psi \in \mathcal{X}$ has multiple extensions to $B$,
\item every $\phi \in \mathcal{Y}$ extends uniquely to a pure state $\tilde{\phi}$ of $B$,
\item given $\phi_0, \phi_1 \in \mathcal{Y}$, then $\phi_0 E \phi_1$ if and only if
$\tilde{\phi}_0 \sim \tilde{\phi}_1$,
\item either of the following conditions can be obtained:
\begin{enumerate}
\item if $A$ is nuclear and there exists a copy of $M_2(\C)$
inside $A^{\mathcal{U}} \cap A'$ for some free ultrafilter $\mathcal{U}$,
then $B$ is nuclear, there exists a copy of $M_2(\C)$
inside $B^{\mathcal{U}} \cap B'$ for some free ultrafilter $\mathcal{U}$, and the restriction map $r: T(B) \to T(A)$ is
a homeomorphism,
\item the restriction map $r: T(B) \to T(A)$ is not injective.
\end{enumerate}
\end{enumerate}
\end{lemma}
\begin{proof}
This lemma can be proved as lemma 2.3 of \cite{farahilan} by
substituting all the instances of
theorem \ref{KOS2} with theorem \ref{thrm:a}.
\end{proof}

Once lemma \ref{lemma:+} is proved, theorem \ref{thrm:2} in the introduction follows from
the proof of lemma 2.8 and theorem 1.2 in \cite{farahilan},
by substituting all instances of lemma 2.3 of \cite{farahilan}
with our lemma \ref{lemma:+}. Point 2 of the
last clause of
theorem \ref{thrm:2} is a consequence of the following fact.

\begin{proposition} \label{prp1}
Let $\set{A_\beta}_{\beta < \aleph_1}$ be an increasing continuous $\aleph_1$-sequence as in
proposition \ref{prop:0} and let $A$ be the inductive limit of the $\aleph_1$-sequence. Suppose that
the set $\set{ \beta < \aleph_1:
r_{\beta +1, \beta} : T(A_{\beta +1}) \to T(A_\beta) \ \text{is not injective}}$ is unbounded in $\aleph_1$.
Then $T(A)$ is nonseparable.
\end{proposition}
\begin{proof}
Suppose $T(A)$ is separable and let $\set{\tau_n}_{n \in \N}$ be a countable
dense subset of $T(A)$.
\begin{claim}
The set $B=\set{ \beta < \aleph_1 : \exists n \ \text{s.t.} \ \tau_n\restriction{A_\beta} \
\text{has multiple extensions to} \ A}$ is unbounded in $\aleph_1$.
\end{claim}
\begin{proof}
Suppose the claim is false and let $\gamma < \aleph_1$ be an upper bound for $B$. Then
each $\tau_n \restriction{A_\gamma}$ has a unique extension to $A_{\gamma +1}$, which,
as we already know from the proof of proposition \ref{prop:0},
is defined through the conditional expectation. If $\gamma$ is big enough
there is a trace $\sigma \in T(A_{\gamma +1})$, $a \in A_\gamma$, and $g \in G_\beta$
such that $\sigma(a u_g) \not= 0$. If $\epsilon > 0 $ is small enough, then
$\set{\tau_n \restriction{A_{\gamma+1}}} \cap \set{ \tau \in T(A_{\gamma+1}) :
\lvert \tau(au_g) - \sigma(au_g) \rvert < \epsilon}$ is empty.
This is a contradiction since
$\set{\tau_n \restriction{A_{\gamma+1}}}$ is dense in $T(A_{\gamma+1})$.
\end{proof}
The claim entails that there is an $\aleph_1$-sequence of traces (modulo taking a cofinal subsequence
of the algebras $A_\beta$)
$\set{ \tau_{\beta}}_{\beta < \aleph_1}$ such that
\begin{enumerate}
\item $\tau_\beta \in T(A_\beta)$ for all $\beta < \aleph_1$,
\item $\tau_\gamma \restriction{A_\beta} = \tau_\beta$ for all $\gamma > \beta$,
\item the trace $\tau_\beta$ admits two different
extensions to $T(A_{\beta+1})$ for every $\beta < \aleph_1$.
\end{enumerate}
This allows to build a discrete set of size $\aleph_1$ in $T(A)$ as follows, which is a contradiction.
For any $\beta <
\aleph_1$ consider $\tau'_{\beta + 1} \in T(A_{\beta +1})$
different from $\tau_{\beta +1}$ and extending $\tau_\beta$,
and pick two open sets in $T(A_{\beta + 1})$ dividing them. Their preimage
via the restriction map $r_{\beta +1}: T(A) \to T(A_{\beta +1})$
are two open disjoint subsets of $T(A)$ such that only one of them contains all the extensions of
$\tau_{\beta +1}$. Hence, any $\aleph_1$-sequence of extensions in $T(A)$ of the
elements in $\set{\tau'_\beta}_{\beta <
\aleph_1}$ has the required property.
\end{proof}
The next two sections are devoted to the proof of theorem \ref{thrm:a}.

\section{Paths of unitaries} \label{sctn:3}
The aim of this section is to prove lemmas \ref{lemma:2+} and \ref{lemma:1+}, two variants
of lemma 2.2
of \cite{KOS} (for simple \cstar-algebras), which are needed for theorem \ref{thrm:a}.
The reader can safely assume these lemmas as blackboxes and
go directly to section \ref{sctn:4}, to see how they are used in the main proofs,
before going through this section.

\begin{lemma} \label{lemma:2+}
Let $A$ be a separable, simple, unital \cstar-algebra, $(\phi_h)_{h \le m}$
some inequivalent pure states and $\set{\tau_1, \dots, \tau_n} \subseteq \partial T(A)$.
For every $F \Subset A$ and $ \epsilon > 0$, there exist
$G \Subset A$ and $\delta >0$ such that, if $(\psi_h)_{h \le m}$ are
pure states which satisfy $\psi_h \approx_{G, \delta} \phi_h$ for
all $1 \le h \le m$,
then for every finite $K \subset A$ and every $\epsilon' > 0$ there is a path of unitaries $(u_t)_{t \in [0,1]}$
such that
\begin{enumerate}
\item $ u_0 = 1$,
\item $\phi_h \circ \text{Ad}(u_1) \approx_{K, \epsilon'} \psi_h$ for all $1 \le h \le m$,
\item $\lVert b - \text{Ad}(u_t)(b) \rVert < \epsilon$ for all $b \in F$,
\item
\begin{enumerate}
\item \label{item4a} $\lVert u_t - 1 \rVert_{2,k} < \epsilon$ for all $k \le n$,\footnote{
We suppress the notation and denote $\lVert \ \rVert_{2, \tau_k}$ by $\lVert \ \rVert_{2,k}$.}
\item \label{item4b} If moreover $A$ is nuclear and
$M_2(\C)$
embeds into $A^{\mathcal{U}} \cap A'$, then $\lVert u_t - 1 \rVert_{2,u} < \epsilon$.
\end{enumerate}
\end{enumerate}
\end{lemma}

\begin{lemma} \label{lemma:1+}
Let $A$ be an infinite-dimensional separable, nuclear, simple, unital, \cstar-algebra such that $M_2(\C)$
embeds into $A^{\mathcal{U}} \cap A'$ and let $(\phi_h)_{h \le m}$ be inequivalent pure states.
For every $F \Subset A$ and $ \epsilon > 0$, there exist
$G \Subset A$ and $\delta >0$ such that, if $(\psi_h)_{h \le m}$ are inequivalent
pure states which satisfy $\psi_h \approx_{G, \delta} \phi_h$ for
all $1 \le h \le m$,
then for every $\sigma \in \mathcal{U}(A)$, $k \in \N\setminus \set{0}$, $K \Subset A$ and every $\epsilon' > 0$
there is a path of unitaries $(u_t)_{t \in [0,1]}$
and an $ a \in A^1$ such that, for every $\tau \in T(A)$
\begin{enumerate}
\item \label{n:i1} $ u_0 = 1$,
\item $\phi_h \circ \text{Ad}(u_1) \approx_{K, \epsilon'} \psi_h$ for all $1 \le h \le m$,
\item \label{n:i2} $\lVert b - \text{Ad}(u_t)(b) \rVert < \epsilon$ for all $b \in F$,
\item \label{n:i4} $\lVert \text{Ad}(\sigma)(a) - \text{Ad}(u_1^{*k})(a) \rVert_{2, \tau} > 1/4$.
\end{enumerate}
\end{lemma}

The proof of lemma \ref{lemma:2+} mirrors the one of proposition 2.2
in \cite{KOS}, with the addition that the unitary $u_1$
is close to the identity with respect of the $\ell_2$-norm induced
by some traces.
In a certain sense, the construction in lemma \ref{lemma:1+} achieves the opposite.
In fact, in this second case, we need a path of unitaries as in
proposition 2.2 of \cite{KOS} so that $u_1$ (or one of its powers) is far from the scalars
with respect of the $\ell_2$-norm induced by a trace.

\subsection{Proof of lemma \ref{lemma:2+}}
We briefly introduce some notation for the following proposition. Given a state $\phi$ on
a \cstar-algebra $A$, we let $L_\phi$ be the following closed left ideal
\[
L_\phi:= \set{ a \in A : \phi(a^*a) = 0} = \set{ a \in A: \pi_\phi(a) \xi_\phi = 0}
\]
We recall that for any state $\phi$ the intersection $L_\phi \cap L_\phi^*$ is a
hereditary subalgebra of $A$.
\begin{proposition} \label{lemma:1}
Let $A$ be an infinite-dimensional, simple, unital \cstar-algebra, $\tau \in T(A)$
and $\phi_1, \dots, \phi_m$ some pure states of $A$.
Then
\[
M = \set{ a \in A : \pi_{\phi_j}(a) \xi_{\phi_j} = \pi_{\phi_j}(a^*) \xi_{\phi_j} = 0 \ \forall j \le m}
\]
is a hereditary subalgebra of $A$ and $\pi_\tau[M]$ is strongly dense in $\pi_\tau[A]''$.
\end{proposition}
\begin{proof}
Since $M= \cap_{ j \le m} L_{\phi_j} \cap L_{\phi_j}^*$,
the strong closure of $\pi_\tau(M)$ is a hereditary subalgebra of $\pi_\tau(A)''$,
therefore it is of the form $p\pi_\tau(A)''p$ for some projection $p \in \pi_\tau(A)''$.
Suppose $p$ is not the identity and let $\eta \in H_\tau$ be a unit vector orthogonal
to the range of $p$. Consider the state $\psi(a) = \langle \pi_\tau(a) \eta, \eta \rangle$.
By uniqueness of the GNS representation, $(\pi_\psi, H_\psi, \xi_\psi)$
is unitarily equivalent to $(\pi_\tau, \overline{\pi_\tau(A) \eta}, \eta)$.
Since $\pi_\tau(A)''$ is a type $\text{II}_1$ von Neumann algebra, the same is true for $\pi_\psi(A)''$
(see proposition 5.3.5 of \cite{dixmier}). 
Consider $a \in \cap_{j \le m} L_{\phi_j}$. Then $a^*a \in M$ and this implies
\[
\lVert \pi_\tau(a) p^{\perp} \rVert ^2 = \lVert p^{\perp} \pi_\tau(a^* a) p^{\perp} \rVert = 0
\]
hence $\pi_\tau(a) \eta = 0$, which means $\pi_\psi (a) \xi_\psi = 0$, which in turn entails
$L_\psi \supseteq \cap_{j \le m} L_{\phi_j}$. Consider the state $\phi = \sum_{j \le m} \frac{1}{m} \phi_j$,
which is such that $L_\phi = \cap_{j \le m} L_{\phi_j}$.
By the correspondence between closed left ideals and weak*-closed faces of $S(A)$ (see theorem
3.10.7 of \cite{ped}\footnote{Here we can consider faces of $S(A)$ instead of $Q(A)$ since $A$
is unital.}) we infer that $\psi$ is contained in the smallest weak*-closed face of $S(A)$ which
contains $\phi$, which is in fact the set
\[
\set{ \theta \in S(A) : \theta(L_\phi) = 0}
\]
On the other hand, the smallest face of $S(A)$ containing the state $\phi$ is
\[
F_\phi  = \set{ \theta \in S(A) : \exists \lambda > 0 \ \theta \le \lambda \phi}
\]
By the Radon-Nikodym theorem (theorem 5.1.2 in \cite{murphy}),
for every state $\theta$ contained
in $F_\phi$, the GNS representation $(\pi_\theta, H_\theta)$ is
(unitarily equivalent to) a subrepresentation of
$(\pi_\phi, H_\phi)$. Since the latter representation is type I (it is in fact the subrepresentation
of a direct sum of irreducible representations),
we get to a contradiction if we can prove that $F_\phi$ is weakly*-closed,
since this would imply that $(\pi_\psi, H_\psi)$ is type I. By Radon-Nikodym
theorem the map
\begin{align*}
\Theta_\phi : \pi_\phi(A)' &\to A^* \\
v & \mapsto \langle \pi_\phi( \ ) v \xi_\phi, \xi_\phi \rangle
\end{align*}
is a linear map such that $\Theta_\phi(\pi_\phi(A)') \cap S(A) = F_\phi$.
Let $\pi$ denote $\oplus_{ i \le m} \pi_{\phi_i}$. We
prove that $\pi(A)'$ is finite-dimensional, which entails that also
$\pi_\phi(A)'$ is finite-dimensional, since $\pi_\phi(A)' = q \pi(A)' q$ for some projection
$q \in \pi(A)'$.
This follows from the contents of Chapter 5 of \cite{dixmier}.
More specifically, if $\phi_1,\dots ,\phi_n$ are equivalent pure states, given
$\pi' = \oplus_{ i \le n} \pi_{\phi_i}$, then $\pi'(A)'$ is a type $\text{I}_n$-factor
by proposition 5.4.7 of \cite{dixmier}, thus it is finite dimensional. 
By theorem 3.8.11 \cite{ped}, the commutant $\pi(A)'$ is therefore the direct sum
of a finite number of finite-dimensional type I factors.
\end{proof}

\begin{corollary} \label{corollary:B}
Let $A$ be an infinite-dimensional, simple, unital \cstar-algebra and $\tau \in T(A)$. Let
$\set{(\pi_i, H_i)}_{i \le n}$ be inequivalent irreducible representations, $F_i \subset H_i$
finite sets and $T_i \in \mathcal{B}(H_i)$ projections. For every $\epsilon > 0$ there is
$b \in A_+$ such that
\begin{enumerate}
\item $\lVert b \rVert < 9$,
\item $\pi_i(b) \restriction_{F_i} \approx_{\epsilon}
T_i \restriction_{F_i}$ for all  $i \le n$,
\item
\begin{enumerate}
\item 
$\lVert \label{l:i3a} b \rVert_{2,\tau} < \epsilon$;
\item \label{l:i3b} If moreover $A$ is nuclear and
$M_2(\C)$
embeds into $A^{\mathcal{U}} \cap A'$, then $\lVert b \rVert_{2,u} < \epsilon$. 
\end{enumerate}
\end{enumerate}
\end{corollary}
\begin{proof}
By the Glimm-Kadison transitivity theorem (see \cite[Corollary 7]{glimm-kadison})
there is $b' \in A_{sa}$ whose norm is smaller than 3/2 such that, for all $i \le n$
\[
\pi_i(b') \restriction_{F_i \cup T_i[F_i]} = T_i \restriction_{F_i \cup T_i[F_i]}.
\]
Define for each $i \le n$ the set
\[
L_i = \set{ a \in A: \pi_i(a) \xi = 0 \ \forall \xi \in F_i \cup T_i[F_i]}.
\]
Let $L$ be the intersection of all $L_i$'s. For every $\tau \in T(A)$, by proposition
\ref{lemma:1} the set $\pi_\tau[L]$ is strongly dense in $\pi_\tau[A]''$.
There is thus a self-adjoint $a_\tau \in L$ such that
$\lVert b' - a_\tau \rVert_{2,\tau} < \epsilon/6$. For every $\tau \in T(A)$ define $ b'_\tau:=b' -
a_\tau$. Notice that $\pi_i(b'_\tau) \restriction_{F_i \cup T_i[F_i]} =
T_i \restriction_{F_i  \cup T_i[F_i]}$ for all $i \le n$, and that $\lVert b'_\tau \rVert_{2, \tau} < \epsilon$,
so the lemma for item \ref{l:i3a} has been proved. From now on, we assume that $A$ is nuclear and that
$M_2(\C)$
embeds into $A^{\mathcal{U}} \cap A'$, and thus that $A$ has CPoU, by theorem \ref{cpou}.
By Kaplansky's density theorem we can assume that $\lVert b'_\tau \rVert < 3$ for every $\tau \in T(A)$.
Let $b_\tau$ be the square of $b'_\tau$. We have that, for every $\tau \in T(A)$
\begin{enumerate}[i.]
\item $b_\tau \in A_+$,
\item $\lVert b_\tau \rVert < 9$,
\item $\lVert b_\tau \rVert_{2,\tau} < \epsilon/2$,
\item $\pi_i(b_\tau) \restriction_{F_i} = T_i \restriction_{F_i }$ for all $i \le n$ (this is the case since all $T_i$'s
are projections).
\end{enumerate}
By compactness of $T(A)$ there exists a finite
set $\set{b_j}_{j \le k} \subset A_+$ of elements whose norm is smaller than $9$ such that
$\sup_{\tau \in T(A)} \min_{j \le k} \lVert b_j \rVert_{2, \tau} < \epsilon/2$ and $\pi_i(b_j) \restriction_{F_i} = T_i \restriction_{F_i}$ for every $j \le k$ and $i \le n$.
Since $A$ has CPoU, for every $\gamma >0$
we can find some contractions $e_1, \dots, e_k \in A_+$ such that
\begin{enumerate}
\item $\sum_{j \le k} e_j = 1$,
\item for every $\tau \in T(A)$ and all $j,h \le k$
\begin{enumerate}
\item $\lVert[e_j, b_h] \rVert < \gamma$,
\item $\lVert e_j e_h \rVert_{2, \tau} < \gamma$,
\item $\lVert e_j^2 - e_j \rVert_{2, \tau} < \gamma$,
\item $\tau(e_j b^2_j) < (\epsilon^2/4)\tau(e_j) + \gamma$.
\end{enumerate}
\end{enumerate}
Let $b:= \sum_{j \le k} e_j^{1/2} b_j e_j^{1/2}$ for $\gamma > 0$ small enough ($\gamma$
also depends on $k$) so that
\begin{enumerate}[i.]
\item $\lVert b \rVert < 9$,
\item $\pi_i(b) \restriction_{F_i} \approx_{\epsilon} \sum_{j \le k} \pi_i(e_j) \pi_i(b_j) \restriction_{F_i} =
\sum_{j \le k} \pi_i(e_j) T_i \restriction_{F_i} = T_i \restriction_{F_i}$ for all $i \le n$,
\item for every $\tau \in T(A)$
\[
\lVert b \rVert_{2,\tau}^2 = \tau (\sum_{j,h \le k} e_jb_jb_h e_h) \approx_{f(\gamma)}
\tau( \sum_{j \le k} e_j b^2_j) < \epsilon^2/4 \sum_{j \le k} (\tau(e_j) + \gamma) < \epsilon^2.
\]
\end{enumerate}
\end{proof}

The following proposition is implicitly used in \cite[Theorem 3.1]{KOS}. We give here a full proof of it.
\begin{proposition} \label{prop:B2}
For every $\epsilon > 0$ and $M \in \N$ there is $\delta >0$ such that the following holds.
Suppose $\xi$ is a norm one vector in an infinite-dimensional Hilbert space $H$, and that
$\set{b_j}_{j \le M} \subseteq B(H)$ are such that $\sum_{j } b_j b^*_j \le 1$
and $\sum_{j} b_j b^*_j \xi = \xi$. Let moreover $\eta \in H$ be a unit vector orthogonal
to the linear span of $\set{b_j b_k^* \xi : j,k \le M}$ such that, for all $j,k \le M$
\[
\lvert \langle b_k^* \xi , b_j^* \xi \rangle - \langle b_k^* \eta , b_j^* \eta \rangle \rvert < \delta
\]
Then there is a projection $q \in B(H)$ such that
\[
\sum_{j \le M} b_j q b_j^* (\eta + \xi) \approx_{\epsilon} 0 \ \ \text{and} \ \
\sum_{j \le M} b_j q b_j^* (\eta - \xi) \approx_{\epsilon} \eta - \xi
\]
\end{proposition}
\begin{proof}
By lemma 3.3 of \cite{FKK}, for every $\epsilon' > 0$ and $M' \in \N$ there is a $\delta' > 0$
such that if $(\xi_1, \dots, \xi_{M'})$ and $(\eta_1, \dots, \eta_{M'})$ are two sequences of vectors
in a Hilbert space $H$ such that $\sum_i \lVert \xi_i \rVert^2 \le 1$,
$\sum_i \lVert \eta_i \rVert^2 \le 1$, and
\[
\lvert \langle \xi_i , \xi_j \rangle - \langle \eta_i, \eta_j \rangle \rvert < \delta ' \ \forall i,j \le M'
\]
then there is a unitary $U \in B(H)$ such that
\[
\lVert U\xi_j - \eta_j \rVert < \epsilon' \ \forall j \le M'
\]
Moreover, if $H$ is infinite dimensional and $\langle \xi_i , \eta_j \rangle =0$ for all
$i, j \le M'$, then $U$ can be chosen to be self-adjoint.
Let $\delta > 0$ be smaller than $\epsilon /M$ and than the $\delta'$
given by lemma 3.3 of \cite{FKK}
for $M'=M$ and $\epsilon' =\epsilon /M$. Fix
$\xi$, $\eta$ and $b_j$ for $j \le M$ as in the statement of the current proposition.
Since the linear spans of $\set{b_j^* \xi : j \le M}$ and $\set{b_j^* \eta : j \le M}$
are orthogonal, there is a self-adjoint unitary
$w$ on $H$ such that, for every $j \le M$
\[
\lVert w b_j^* \xi - b_j^* \eta \rVert < \epsilon /2M
\]
\[
\lVert w b_j^* \eta - b_j^* \xi \rVert < \epsilon /2M
\]
This entails, since $\lVert b_j \rVert \le 1$ for all $j \le M$, $\lVert b_j w b_j^* \xi - b_j b_j^* \eta
\rVert < \epsilon /2M$, therefore
\[
\lVert \sum_{ j \le M} b_j w b_j^* \xi - \sum_{j \le M} b_j b_j^* \eta \rVert < \epsilon /2
\]
Similarly we have
\[
\lVert \sum_{ j \le M} b_j w b_j^* \eta - \sum_{j \le M} b_j b_j^* \xi \rVert < \epsilon /2
\]
Moreover $\sum_{j} b_j b^*_j \xi = \xi$ and $\delta < \epsilon /M$ imply
$\sum_{j} b_j b^*_j \eta \approx_\epsilon \eta$. Thus, if $q$ is the projection $(1- w)/2$,
it follows that
\[
\sum_{j \le M} b_j q b_j^* (\eta + \xi) \approx_{\epsilon} 0 \  \ \text{and} \ \
\sum_{j \le M} b_j q b_j^* (\eta - \xi) \approx_{\epsilon} \eta - \xi
\]
\end{proof}

\begin{proposition} \label{prop:B}
For every $\epsilon >0$ and $N > 0$ there exists $\delta > 0 $ such that for every
self-adjoint element $a$ of norm smaller than $N$ on a Hilbert space $H$, every $r \in [-N, N]$,
and all unit vectors $\xi \in H$, we have the following. If $ r \xi \approx_\delta a \xi $
then $ \exp(i \pi r) \xi \approx_\epsilon \exp(i \pi a)\xi$.
\end{proposition}
\begin{proof}
Fix $\epsilon, N >0$ and
let $p(x)$ be a polynomial such that
\[
\lVert (p(x) - \exp(i \pi x))_{\restriction [-N,N]} \rVert_\infty < \epsilon/3
\]
It is straightforward to find $\delta > 0$ (depending only on $\epsilon, N$ and $p(x)$)
such that $a \xi \approx_{\delta} r \xi$ implies
$p(r) \xi \approx_{\epsilon/3} p(a) \xi$. Thus we have
\[
\exp(i \pi r) \xi \approx_{\epsilon/3} p(r) \xi \approx_{\epsilon/3} p(a) \xi
\approx_{\epsilon/3} \exp(i \pi a)\xi
\]
\end{proof}

\begin{proof}[Proof of lemma \ref{lemma:2+}]
It is sufficient to show the following claim.
\begin{claim}
Let $A$ be a separable simple unital \cstar-algebra, $(\phi_h)_{h \le m}$
some inequivalent pure states and $\set{\tau_1, \dots, \tau_n} \subseteq \partial T(A)$.
For every finite $  F \Subset A$ and $ \epsilon > 0$, there exist a finite
$G \Subset A$ and $\delta >0$ such that the following holds.
Suppose $(\psi_h)_{h \le m}$ are pure states such that
$\psi_h \sim \phi_h$, and that moreover $\psi_h \approx_{G, \delta} \phi_h$ for
all $1 \le h \le m$.
Then there exists a path of unitaries $(u_t)_{t \in [0,1]}$ in $A$
satisfying the following
\begin{enumerate}
\item \label{lemma:i1} $ u_0 = 1$,
\item $\phi_h \circ \text{Ad}(u_1)= \psi_h$ for all $1 \le h \le m$,
\item $\lVert b - \text{Ad}(u_t)(b) \rVert < \epsilon$ for all $b \in F$,
\item 
\begin{enumerate}
\item \label{lemma:i3a} $\lVert u_t - 1 \rVert_{2,k} < \epsilon$ for all $k \le n$;
\item \label{lemma:i3b} If moreover $A$ is nuclear and
$M_2(\C)$
embeds into $A^{\mathcal{U}} \cap A'$, then $\lVert u_t - 1 \rVert_{2,u} < \epsilon$.
\end{enumerate}
\end{enumerate}
\end{claim}
Given $m \in \N$ and $\vec{\phi} \in P(A)^m$, the set
\[
\set{ \vec{\phi} \circ \text{Ad}(u) : u \in \mathcal{U}(A)}
\]
is weak* dense in $P(A)^m$ (see \cite[Lemma 2.3]{FKK}). Hence, the statement of lemma \ref{lemma:2+} follows from the claim (see \cite[Lemma 2.2]{KOS} for details).

We start by showing the existence of a path of unitaries satisfying
items \ref{lemma:i1}-\ref{lemma:i3a} of the claim.
By an application of the Glimm-Kadison transitivity theorem, there exists
$\epsilon'' > 0 $ such that if $(\theta_h)_{h \le m}$ are inequivalent pure states
and $(\chi_h)_{h \le m}$ are pure states such that $\lVert \theta_h - \chi_h \rVert < \epsilon''$,
then there is a path of unitaries $(v_t)_{t \in [0,1]}$
which satisfies the following
\begin{itemize}
\item $\theta_h \circ \text{Ad}(v_1) = \chi_h$ for all $ 1 \le h \le m$,
\item $\lVert v_t - 1 \rVert < \epsilon/8 $ for all $t \in [0,1]$.
\end{itemize}
In fact for every $h \le m$, if $\lVert \theta_h - \chi_h \rVert$ is small enough, $\theta_h$ and $\chi_h$
are two vector states on $H_{\theta_h}$ induced by two vectors $\xi_{\theta_h}$ and $\zeta_{\chi_h}$
which can be chosen to be very close (depending on $\lVert \theta_h - \chi_h \rVert$). Hence
there is $u_h \in \mathcal{U}(B(H_{\theta_h}))$ which sends $\xi_{\theta_h}$ to $\zeta_{\chi_h}$ and is very
close to the identity of $B(H_{\theta_h})$, which in turn implies that $u_h = \exp(ia_h)$ 
for some $a_h \in B(H_{\theta_h})_{sa}$ whose norm is close to zero. Given
the representation $\pi = \bigoplus_{h \le m}
\pi_{\theta_h}$ on $H = \bigoplus_{h \le m} H_{\theta_h}$, by Glimm-Kadison transitivity theorem
there is $b \in B(H)_{sa}$ which
behaves like $a_h$ on $\xi_{\theta_h}$ for every $h \le m$, and whose norm is close zero.
The required path is $(v_t)_{t \in [0,1]}$, where $v_t = \exp (itb)$. Fix such $\epsilon''$.

Let $\epsilon' > 0$
be smaller than the $\delta$ provided by proposition \ref{prop:B} for $N=2^{2n}$ and
$\min \set{\epsilon''/2, \epsilon/4}$.
Let $(\pi_h, H_h, \xi_h)$ be the GNS representations associated to $\phi_h$, let
$(\pi, H)$ be the direct sum of them, and let $p \in B(H)$ be the projection onto the span
of the cyclic vectors $\xi_h$ for $h \le m$. The representation $\pi$
has an approximate diagonal since it is the direct sum of some inequivalent irreducible
representations (see section 4 of \cite{KOS}), thus there is a positive
integer $M$ and some $b_j \in A$ for $j \le M$ such that
\begin{itemize}
\item $\sum_j b_j b_j^* \le 1$,
\item $p ( 1- \sum_j \pi(b_j b_j^*)) = 0$,
\item $\sup_{c \in A, \lVert c \rVert \le 1} \lVert b \sum_j b_j c b_j^* - \sum b_j c b_j^* b \rVert <
\frac{\epsilon}{4} \frac{1}{e^{\pi 2^{2n}}2^{2n} }$ for all $b \in F$.
\end{itemize}
Fix $\delta = \delta'/2$, $\delta'$ being the value given by proposition \ref{prop:B2} for $M$
and $\epsilon'$.
Fix moreover
\[
G = \set{b_j b_k^* : j,k \le M}
\]
Suppose $\psi_h \sim \phi_h$ and $\psi_h \approx_{G, \delta} \phi_h$ for all $h \le m$.
For every $h \le m$
pick $w_h \in \mathcal{U}(A)$ such that $\phi_h \circ \text{Ad}(w_h) =
\psi_h$, and let $\eta_h$ denote the vector $w_h \xi_h$. By Glimm's lemma
(see lemma 1.4.11 in \cite{brownozawa}) there are, for every $h \le m$, $\zeta_h \in H_h$
unit vectors orthogonal to $\set{ \pi(b_j b_k^*) \xi_h, \pi(b_j b^*_k) \eta_h : j,k \le M}$
such that, if $\theta_h = \omega_{\zeta_h} \circ \pi_h$, we have $\theta_h \approx_{G, \delta}
\psi_h$ for every $h \le m$. As a consequence $\theta_h \approx_{G, \delta'}
\phi_h$ for all $h \le m$, which implies, for $j,k \le M$
\[
\lvert \langle \pi(b_k)^* \xi_h, \pi(b_j)^* \xi_h \rangle  - 
 \langle \pi(b_k)^* \zeta_h, \pi(b_j)^* \zeta_h \rangle \rvert < \delta'
\]
From an application of proposition \ref{prop:B2} for $\xi = \xi_h$,
$\eta = \zeta_h$ and $b_j = \pi_h(b_j)$, we obtain a projection $q_h \in B(H_h)$
such that $v_h = \exp(i \pi \sum_j b_j q_h b^*_j)$
satisfies $\zeta_h \approx_{\epsilon''/2} v_h \xi_h$.
By Glimm-Kadison transitivity theorem there is $a \in A^1_{sa}$ which agrees with $q_h$ on
$S_h = \text{span} \set{\pi(b^*_j)\xi_h, \pi(b^*_j)\zeta_h, \pi(q_h)\pi(b^*_j)\xi_h, \pi(q_h)\pi(b^*_j)\zeta_h
: j \le M}$ for every $h \le m$. For each $k \le n$ corollary \ref{corollary:B} (item \ref{l:i3a}) provides
one $a_k \in A_{sa}$ such that $\lVert a_k \rVert_{2,k} \le {\epsilon'}^2/(9^{4n} M)$,
which moreover agrees with $q_h$ on $S_h$ for all $h \le m$. We can assume moreover
that each $a_k$ has norm smaller than 9. Define
$\overline{a}$ to be the sum $\sum_j b_j a_1 \dots a_n a^2 a_n \dots a_1 b^*_j$.
This is a positive element whose norm is smaller than $2^{2n}$. Define
$u_t$ for $t \in [0,1]$ to be
$\exp ( i t \pi \overline{a})$.
Thus, combining proposition \ref{prop:B} with the previous construction, we get 
$\lVert \pi(u_1) \xi_h - \zeta_h \rVert < \epsilon''/2$ for all $h \le m$.
This implies $\lVert \phi_h - \theta_h \rVert < \epsilon''$.
Moreover for all $b \in F$ we have
\[
\lVert [u_t, b] \rVert \le e^{\pi \lVert \overline{a} \rVert} \lVert [\overline{a}, b] \rVert
\le \epsilon /4
\]
Finally, let $\tilde{a}_k$ be $a_k / \lVert a_k \rVert$. Then for each $k \le n$ we can show that
\[
\tau_k(\overline{a}^2) \le 9^{4n} \sum_{j \le M} \tau_k (b_j \tilde{a}_1 \dots \tilde{a}_n a^2 \tilde{a}_n \dots
\tilde{a}_1 b^*_j) =
\]
\[
= 9^{4n} \sum_{j \le M} \tau_k(\tilde{a}_k \dots \tilde{a}_n a^2 \tilde{a}_n \dots \tilde{a}_1 b^*_j b_j
\tilde{a}_1 \dots \tilde{a}_{k-1}) \le
\]
\[
\le 9^{4n} \sum_{j \le M} [\tau_k((\tilde{a}_k \dots \tilde{a}_n a^2 \tilde{a}_n \dots \tilde{a}_k)^2]^{1/2}
[\tau_k((\tilde{a}_{k-1} \dots \tilde{a}_1 b^*_j b_j \tilde{a}_1 \dots \tilde{a}_{k-1})^2]^{1/2} \le
\]
\[
\le 9^{4n} \sum_{j \le M} \tau_k(\tilde{a}_k \dots \tilde{a}_n a^2 \tilde{a}_n \dots \tilde{a}_k)^{1/2}
\le 9^{4n} \sum_{j \le M} \tau_k(\tilde{a}_k^2)^{1/2} \le {\epsilon'}^{2}
\]
Therefore $\lVert \overline{a} \rVert_{2,k} \le \epsilon'$, thus $\lVert u_t -1 \rVert_{2,k}
\le \epsilon/4$.
Perform the same construction between $(\theta_h)_{h \le m}$ and $(\psi_h)_{h \le m}$
to find a path of unitaries $(v_t)_{t \in [0,1]}$ such that
$\lVert \psi_h \circ \text{Ad}(v_1) - \theta_h \rVert < \epsilon''$ for all $h \le m$, 
$\lVert [v_t, b] \rVert  \le \epsilon /4$ for all $t \in [0,1]$ and $b \in F$ and finally such that
$\lVert v_t -1 \rVert_{2,k} \le \epsilon/4$ for all $t \in [0,1]$ and $k \le n$.
The argument sketched at the
beginning of the proof allows to find two paths of untaires $(u'_t)_{t \in [0,1]}$, $(v'_t)_{t \in [0,1]}$
such that $\phi_h \circ \text{Ad} (u_1 u'_1) = \theta_h$ and
$\psi_h \circ \text{Ad} (v_1 v'_1) = \theta_h$ for all $h \le m$, and such that
$\lVert u'_t - 1 \rVert < \epsilon/8$, $\lVert v'_t - 1 \rVert < \epsilon/8$ for all $t \in [0,1]$.
Then $(u_t u'_t v'^*_t v^*_t)_{t \in [0,1]}$ is the desired path.

In order to build a path of unitaries satisfying item \ref{lemma:i3b} of the claim, the first part of the proof is analogous to what we just did.
Instead of using item \ref{l:i3a} corollary \ref{corollary:B} to find $a_k \in A_{sa}$,
by item \ref{l:i3b} of the same lemma, for every $\tau \in T(A)$, there
 is $a \in A_{+}$ with norm smaller than 9 such that
$\lVert a \rVert_{2, \tau} < (\epsilon')^2/9M$ which moreover
agrees with $q_h$ on $S_h = \text{span} \set{\pi(b^*_j)\xi_h, \pi(b^*_j)\zeta_h,
: j \le M}$ well enough so that $\pi(\exp{(i \pi \sum_j b_j a b^*_j)}) \xi_h \approx_{\epsilon''/2} \zeta_h$
for every $h \le m$.
Let $\overline{a}$ be $\sum_j b_j a b^*_j$ and, for $t \in [0,1]$, $u_t = \exp{(it\pi \overline{a})}$.
This implies $\lVert \phi_h \circ \text{Ad}(u_1) - \theta_h \rVert < \epsilon''$.
For all $b \in F$ we have
\[
\lVert [u_t, b] \rVert \le e^{\pi \lVert \overline{a} \rVert} \lVert [\overline{a}, b] \rVert
\le \epsilon /4.
\]
Notice that $\overline{a}$ is positive with norm smaller than $9$.
This implies that, for $\tau \in T(A)$,
\[
\lVert \overline{a} \rVert_{2,\tau}^2 \le
9 \tau(\overline{a})
= 9 \sum_{j \le M} \tau (b_j a b^*_j) = 9 \sum_{j \le M} \tau ( a b^*_jb_j) \le
\]
\[
 \le 9 \sum_{j \le M} \lVert a \rVert_{2, \tau} \lVert b_j^* b_j \rVert_{2,\tau} \le
 9 M \lVert a \rVert_{2, \tau}
\]
Since $\lVert \overline{a} \rVert_{2,\tau} \le \epsilon'$, it follows that $\lVert u_t -1 \rVert_{2,\tau}
\le \epsilon/4$ for every $\tau \in T(A)$.
Perform the same construction between $(\theta_h)_{h \le m}$ and $(\psi_h)_{h \le m}$
to find a path of unitaries $(v_t)_{t \in [0,1]}$ such that
$\lVert \psi_h \circ \text{Ad}(v_1) - \theta_h \rVert < \epsilon''$ for all $h \le m$, 
$\lVert [v_t, b] \rVert  \le \epsilon /4$ for all $t \in [0,1]$ and $b \in F$ and finally such that
$\lVert v_t -1 \rVert_{2,\tau} \le \epsilon/4$ for all $t \in [0,1]$ and $\tau \in T(A)$.
The argument sketched at the
beginning of the proof allows to find two paths of untaires $(u'_t)_{t \in [0,1]}$, $(v'_t)_{t \in [0,1]}$
such that $\phi_h \circ \text{Ad} (u_1 u'_1) = \theta_h$ and
$\psi_h \circ \text{Ad} (v_1 v'_1) = \theta_h$ for all $h \le m$, and such that
$\lVert u'_t - 1 \rVert < \epsilon/8$, $\lVert v'_t - 1 \rVert < \epsilon/8$ for all $t \in [0,1]$.
Then $(u_t u'_t v'^*_t v^*_t)_{t \in [0,1]}$ is the desired path.

\end{proof}

\subsection{Proof of lemma \ref{lemma:1+}}
\begin{proof}[Proof of lemma \ref{lemma:1+}]
We start with the following claim.
\begin{claim} \label{claim:1+}
For every $F' \Subset A$, $\gamma > 0$ and $k' \in \N \setminus \set{0}$ there is $a \in A^1$ and a path
of unitaries $(w_t)_{t \in [0,1]}$ such that
\begin{enumerate}[a.]
\item $w_0 = 1$
\item $ \lVert [b,a] \rVert < \gamma$ and $\max_{t \in [0,1]} \lVert [w_t, b] \rVert < \gamma$ for
all $b \in F'$,
\item $\lVert \text{Ad}(w^{*k'}_1)(a) - a \rVert_{2, \tau} > 3/4$ for all $\tau \in T(A)$.
\end{enumerate}
\end{claim}
\begin{proof}
Let $a' \in M_2(\C)$ be unitary whose trace is zero. Considering the isomorphism $M_{2^\infty} \cong
\bigotimes_{n \in \mathbb{Z}} M_2(\C)$, let $a \in M_{2^\infty}$ the element corresponding to the
tensor
\[
\dots 1 \otimes 1 \otimes a' \otimes  1 \otimes 1 \dots
\]
where $a'$ appears in the coordinate corresponding to zero. Let $\beta \in \text{Aut}(\bigotimes_{n \in \mathbb{Z}} M_2(\C))$ be the automorphism sending the $n$-th copy of $M_2(\C)$ to the
$n+1$-th copy. A direct computation shows that, for $l \in \mathbb{Z} \setminus \set{0}$
\[
\lVert \beta^l(a) - a \rVert_{2, \tau_{M_{2^\infty}}} = \sqrt{2},
\]
where $\lVert \cdot \rVert_{2, \tau_{M_{2^\infty}}}$ is the $\ell_2$-norm induced by $\tau_{M_{2^\infty}}$,
the unique trace of $M_{2^\infty}$.
Since all automorphisms in the \textsf{CAR} algebra are asymptotically inner with
a path starting from the identity (see \cite[Exercise 7.7.5]{blackk}),
there exists a path of unitaries $(w_t)_{t \in [0,1]}$ in $M_{2^\infty}$ such that $w_0 = 1$ and
\[
\lVert \text{Ad}(w^{*k'}_1)(a) - a \rVert_{2, \tau_{M_{2^\infty}}} > 1.
\]
Moreover, since $M_{2^\infty}$ is \textsf{UHF}, we can assume that there is a self-adjont
$d \in M_{2^\infty}$ such that $w_t = \exp{(i td)}$ for all $t \in [0,1]$. By hypothesis and remark
\ref{naim:remark} there exists a copy of $M_{2^{\infty}}$ into $A^{\mathcal{U}} \cap A'$,
hence we can identify $d$ and $a$ with elements $(d_n)_{n \in \N}$ and
$(a_n)_{n \in \N}$ of $A^{\mathcal{U}} \cap A'$. Let $w_{t,n} = \exp(itd_n)$ and, without loss of generality,
assume that $\lVert a_n \rVert = 1$ for all
$n \in \N$. Since $[0,1]$ is compact there exists
$X \in \mathcal{U}$ such that $\lVert [a_n, b] \rVert < \gamma$ and $\max_{t \in [0,1]} \lVert
[w_{t,n}, b] \rVert < \gamma$ for all $n \in X$ and all $b \in F'$.
We claim that there is $n \in X$ such that, for all $\tau \in T(A)$
\[
\lVert \text{Ad}(w_{1,n}^{*k'})(a_n) - a_n \rVert_{2, \tau} > 3/4.
\]
Suppose this is not the case, then for every $n \in X$ we can find a $\tilde{\tau}_n$ for which the inequality
above is false. Let $\vec{\tau} = (\tau_n)_{n \in \N}$ be a trace on $A^{\mathcal{U}}$ such that
$\tau_n = \tilde{\tau}_n$ for all $n \in X$. On the one hand we have that
\[
\lVert \text{Ad}(w^{*k'}_1)(a) - a \rVert_{2, \vec{\tau}} \le 3/4.
\]
On the other, since all traces of $A^{\mathcal{U}}$ restrict to
$\tau_{M_{2^\infty}}$ on $M_{2^\infty}$, we have that
\[
\lVert \text{Ad}(w^{*k'}_1)(a) - a \rVert_{2, \vec{\tau}} = \lVert \text{Ad}(w^{*k'}_1)(a) - a \rVert_{2, \tau_{M_{2^\infty}}} > 3/4,
\]
which is a contradiction.
\end{proof}
Given $F \Subset A$ and $\epsilon > 0$, let $G \Subset A$ and $\delta > 0$ be given
by lemma \ref{lemma:2+}.
Let  $(\psi_h)_{h \le m}$ be inequivalent pure states such that $\psi_h \approx_{G, \delta} \phi_h$ for
all $1 \le h \le m$, and fix $\sigma \in \mathcal{U}(A)$, $k \in \N\setminus \set{0}$, $K \Subset A$ and $\epsilon' > 0$. By lemma \ref{lemma:2+} there
exists a path of unitaries $(v_t)_{t \in [0,1]}$ such that
\begin{enumerate}[i.]
\item $v_0 = 1$,
\item $\phi_h \circ \text{Ad}(v_1) \approx_{K, \epsilon'} \psi_h$ for all $1 \le h \le m$,
\item $\lVert b - \text{Ad}(v_t)(b) \rVert < \epsilon$ for all $b \in F$.
\pushcounter
\end{enumerate}
Apply claim \ref{claim:1+} to $F \cup K \cup \set{\sigma, v_1}$, $k$ and $\gamma > 0 $ small
enough to find $a \in A^1$ and a path of unitaries $(w_t)_{t \in [0,1]}$ such that
\begin{enumerate}[i.]
\popcounter
\item $w_0 = 1$,
\item $\phi_h \circ \text{Ad}(v_1w_1) \approx_{K, \epsilon'} \psi_h$ for all $1 \le h \le m$,
\item $\lVert b - \text{Ad}(v_tw_t)(b) \rVert < \epsilon$ for all $b \in F$,
\item $\lVert \text{Ad}(\sigma)(a) - \text{Ad}((v_1w_1)^{*k})(a) \rVert_{2, \tau} \ge
\lVert a - \text{Ad}(w_1^{*k})(a) \rVert_{2, \tau} - f(\gamma) > 1/2$ for all $\tau \in T(A)$.
\end{enumerate}
\end{proof}

\section{A variant of Kishimoto-Ozawa-Sakai theorem} \label{sctn:4}
In this section we prove theorem \ref{thrm:a}. We split the proof in two parts, the first for
clause \ref{ta:i1}, the second for clause \ref{ta:i2}.

\begin{proof}[Proof of theorem \ref{thrm:a} - part \ref{ta:i1}]
Fix a dense $\set{a_i}_{i \in \N}$ in $A$ and a dense $\set{\sigma_j}_{j \in \N}$ in $\mathcal{U}(A)$. The construction proceeds by induction on
the couples $(j,k) \in \N \times \N$. These two indices keep track of the fact that
we want to build an automorphism $\alpha$ such that
 for every $\tau \in T(A)$ the $k$-th power of
its extension $\alpha^k_{\tau}$ to $\pi_{\tau}[A]''$ is far away from all $\text{Ad}(\sigma_j)$
in the $\ell_2$-norm induced by $\tau$. Let
$\preceq$ be any well-ordering of the couples $\N \times \N$,
and assume that the first elements of such ordering
are $(1,1) \prec (2,1)$.
We shall present in detail step 1 and 2 of the construction, then the generic $n$-th step.
\begin{description}
\item[Step 1]
\begin{itemize}
\item[\emph{a1)}]
By assumption, there is a copy of $M_2(\C)$ in $A^{\mathcal{U}} \cap A'$ for some
free ultrafilter $\mathcal{U}$. Let $c_1, d_1 \in A^{\le 1}$ be two elements given
by clause \ref{ni2} of proposition \ref{prop:comm} for the finite set $\set{a_1}$ and
$\epsilon_1 = 2^{-6}$.
Apply lemma \ref{lemma:1+} to $\phi_1$ for $F_1 = \set{a_1, c_1, d_1}$ and
$\epsilon_1 = 2^{-6}$ to find a $G_1 \Subset A$ and
$\delta_1 >0$
which satisfy the thesis of the lemma.
\item[\emph{b1)}] Fix $\tilde{\psi}_1 \sim \psi_1$ such that $\tilde{\psi}_1
\approx_{G_1, \delta_1} \phi_1$.
\item[\emph{a2)}] Apply lemma \ref{lemma:2+} to $\tilde{\psi}_1$ for $F_1'=F_1$, $\epsilon_1$ to find a $G'_1 \Subset A$ and $\delta'_1> 0$ which
satisfy the thesis of the lemma (specifically satisfying item \ref{item4b}).
\item[\emph{b2)}] Fix $\sigma= \sigma_1$, $k=1$, $K = G'_1 \cup F'_1$ and $\epsilon' = \min \set{\delta'_1, 1/2}$, and let $(v_{1,t})_{
t \in [0,1]}$ be a path of unitaries in $A$ and $b_{1,1} \in A^1$ given by the application of lemma \ref{lemma:1+} in
part \emph{a1} such that (we will denote $v_{1,1}$ simply by $v_1$):
\begin{enumerate}[(i)]
\item $v_{1,0} = 1$,
\item $\phi_1 \circ \text{Ad} (v_1) \approx_{K, \epsilon'} \tilde{\psi}_1$,
\item $\lVert b - \text{Ad}(v_{1,t})(b) \rVert < \epsilon_1$ for all $b \in F_1$,
\item $\lVert \text{Ad} (\sigma_1) ( b_{1,1}) - \text{Ad} (v^*_1) (b_{1,1}) \rVert_{2,\tau} > 1/4$
for all $\tau \in T(A)$.
\end{enumerate}
\end{itemize}
\item[Step 2]
\begin{itemize}
\item[\emph{a1)}]
Let $c_2, d_2 \in A^{\le 1}$ be two elements given
by clause \ref{ni2} of proposition \ref{prop:comm} for the set $F_1' \cup \set{v^*_1}$ and
$\epsilon_2 = 2^{-7}$.
Apply lemma \ref{lemma:1+}
to $\phi_1 \circ \text{Ad} (v_1)$ for $F_2 =F'_1 \cup \set{a_i, c_i, d_i, \text{Ad}(v^*_1)
(a_i) : i \le 2}\cup \set{b_{1,1}, v_1^*}$, and $\epsilon_2 = 2^{-7}$
to find a $G_2 \Subset A$ and $\delta_2 > 0 $
which satisfy the thesis of the lemma.
\item[\emph{b1)}] \label{nb1} Fix $K = G_2 \cup F_2$ and $\epsilon' = \min \set{\delta_2, 1/4}$,
and let $(w_{1,t})_{
t \in [0,1]}$ be a path of unitaries in $A$ given by the application of lemma \ref{lemma:2+} in
part \emph{a2} of the previous step such that (we will denote $w_{1,1}$ simply by $w_1$):
\begin{enumerate}[(i)]
\item $w_{1,0} = 1$,
\item $\phi_1 \circ \text{Ad} (v_1) \approx_{K, \epsilon'} \tilde{\psi}_1 \circ \text{Ad} (w_1)$,
\item $\lVert b - \text{Ad}(w_{1,t})(b) \rVert < \epsilon_1$ for all $b \in F'_1$,
\item $\lVert w_1 - 1 \rVert_{2,\tau} < \epsilon_1$ for all $\tau \in T(A)$.
\end{enumerate}
Let $u_1$ be equal to $w_1 v^*_1$. We have, for $\tau \in T(A)$,
\begin{align*}
\lVert \text{Ad} (\sigma_1)(b_{1,1}) - \text{Ad} (u_1) (b_{1,1}) \rVert_{2,\tau} 
&\stackrel{(\text{2.b1.iv})}{\ge} \lVert \text{Ad} (\sigma_1)(b_{1,1}) -
\text{Ad} (v^*_1) (b_{1,1}) \rVert_{2,\tau} - 2^{-5}\\ &\stackrel{(\text{1.b2.iv})}{>} 1/8.
\end{align*}
Moreover, we have, for $x_1 \in \set{c_1, d_1}$
\[
\lVert \text{Ad}(u_1)(x_1) - x_1 \rVert_{2, \tau} \stackrel{(\text{1.b2.iii})}{\le}
\lVert \text{Ad}(w_1)(x_1) - x_1 \rVert_{2, \tau}
+ \epsilon_1 \stackrel{(\text{2.b1.iii})}{\le} 2 \epsilon_1.
\]
Conclude by fixing $\tilde{\psi}_2 \sim \psi_2$ such that $\phi_2 \circ \text{Ad} (v_1)
\approx_{K, \epsilon'}  \tilde{\psi}_2 \circ \text{Ad} (w_1)$.
\item[\emph{a2)}] Apply lemma \ref{lemma:2+} to $(\tilde{\psi}_1\circ \text{Ad}(w_1),
\tilde{\psi}_2 \circ \text{Ad}(w_1))$ for
$F_2' = F_2 \cup \set{\text{Ad}(w_1^*)(a_i) : i \le 2}$, $\epsilon_2$
to find a $G'_2 \Subset A$ and $\delta'_2 >0$
which satisfy the thesis of the lemma (with item \ref{item4b}).
\item[\emph{b2)}] Fix $\sigma = v_1 \sigma_2$, $k=1$,
$K = G'_2 \cup F'_2$, $\epsilon' = \min \set{\delta'_2, 1/4}$, and let $(v_{2,t})_{
t \in [0,1]}$ be a path of unitaries in $A$ and $b_{2,1} \in A^1$ given by the application of lemma \ref{lemma:1+} in
part \emph{a1} such that (we will denote $v_{2,1}$ simply by $v_2$)
\begin{enumerate}[(i)]
\item $v_{2,0} = 1$,
\item $\phi_h \circ \text{Ad} (v_1 v_2) \approx_{K, \epsilon'} \tilde{\psi}_h \circ \text{Ad} (w_1)$
for $h \le 2$,
\item $\lVert b - \text{Ad}(v_{2,t})(b) \rVert < \epsilon_2$ for all $b \in F_2$,
\item $\lVert \text{Ad} (v_1\sigma_2) ( b_{2,1}) - \text{Ad} (v^*_2) (b_{2,1}) \rVert_{2,\tau} > 1/4$ for all $\tau \in T(A)$.
\end{enumerate}
\end{itemize}
\end{description}
Assume $(j',k')$ is the $n$-th
element of the ordering induced on $\N \times \N$ by $\prec$. We define
$L = \max \set{k : (j,k) \preceq (j',k')}$.
\begin{description}
\item[Step n]
\begin{itemize}
\item[\emph{a1)}]
Let $c_n, d_n \in A^{\le 1}$ be two elements given
by clause \ref{ni2} of proposition \ref{prop:comm} for the set $F_{n-1}'
\cup \set{u_{n-2}v^*_{n-1}}$ and $\epsilon_n = 2^{-(5+n)} / (4L(L+1)^2)$.
Apply lemma \ref{lemma:1+} to $(\phi_h\circ \text{Ad}
(v_1 \dots v_{n-1}))_{h \le n}$ for $F_n = F'_{n-1} \cup
\set{a_i, c_i, d_i, \allowbreak \text{Ad}(v^*_{n-1}  \dots v^*_1)
(a_i) : i \le n}\cup \set{b_{j,k}: (j,k) \prec (j',k')} \cup \set{ u_{n-2}v_{n-1}^* }$ and
$\epsilon_n = 2^{-(5+n)} / (4L(L+1)^2)$ to find a $G_n \Subset A$ and $\delta_n > 0 $
which satisfy the thesis of the lemma.
\item[\emph{b1)}] Fix $K = G_n \cup F_n$ and $\epsilon' = \min \set{\delta_n, 2^{-n}}$,
and let $(w_{n-1,t})_{
t \in [0,1]}$ be a path of unitaries in $A$ given by the application of lemma \ref{lemma:2+} in
part \emph{a2} of the previous step such that (we will denote $w_{n-1,1}$ simply by $w_{n-1}$):
\begin{enumerate}[(i)]
\item $w_{n-1,0} = 1$,
\item $\phi_h \circ \text{Ad} (v_1 \dots v_{n-1}) \approx_{K, \epsilon'}
\tilde{\psi}_h \circ \text{Ad} (w_1 \dots w_{n-1})$ for $h \le n-1$,
\item $\lVert b - \text{Ad}(w_{n-1})(b) \rVert < \epsilon_{n-1}$ for all $b \in F'_{n-1}$,
\item $\lVert w_{n-1} - 1 \rVert_{2,\tau} < \epsilon_{n-1}$ for all $\tau \le T(A)$.
\end{enumerate}
Let $u_{n-1}$ be equal to $u_{n-2}w_{n-1}v^*_{n-1}$. For every $(j, k) \prec (j',k')$ and $\tau \in
T(A)$
we have, assuming that $(j,k)$ corresponds to the $N$-th element of the well-ordering
$\prec$:
\begin{multline*}
\lVert \text{Ad} (\sigma_j)(b_{j,k}) - \text{Ad}(u_{n-1}^k)(b_{j,k}) \rVert_{2,\tau}
\stackrel{b_{j,k}, u_{N-2}v_{N-1}^* \in F_{N+1}, (\text{N--N+1.b1.iv})}{\ge}
\\
\ge  \lVert \text{Ad} (\sigma_j) (b_{j,k}) - \text{Ad}((u_{N-2}
v^*_{N-1}v^*_{N})^k)
(b_{j,k}) \rVert_{2,\tau} - 2^{-4} \stackrel{u_{N-2}v_{N-1}^* \in F_{N}}{\ge}
\\
\ge \lVert \text{Ad}( (u^*_{N-2}v_{N-1})^k\sigma_j)(b_{j,k}) -
\text{Ad}(v^{*k}_N)(b_{j,k}) \rVert_{2,\tau} - 2^{-3} \stackrel{(\text{N.b2.iv})}{>} 1/8.
\end{multline*}
We also have, for $i < n$, $\tau \in \mathcal{T(A)}$ and $x_i \in \set{c_i, d_i}$
\[
\lVert \text{Ad}(u_{n-1})(x_i) - x_i \rVert_{2, \tau}\stackrel{x_i \in F_i, (\text{i-1.b1.iv})}{\le}
\lVert \text{Ad}(u_{i-2}v^*_{i-1})(x_i) - x_i \rVert_{2, \tau} + \epsilon_{i-3} \stackrel{(\text{i.a1})}{<}
\epsilon_{i-4}.
\]
Conclude by fixing $\tilde{\psi}_n \sim \psi_n$ such that $\phi_n \circ \text{Ad} (v_1 \dots v_{n-1})
\approx_{K, \epsilon'}  \tilde{\psi}_n \circ \text{Ad} (w_1  \dots \\ w_{n-1})$.
\item[\emph{a2)}] Apply lemma \ref{lemma:2+} to
$(\tilde{\psi}_h \circ \text{Ad} (w_1 \dots w_{n-1}))_{h \le n}$ for
$F'_n = F_n \cup   \set{\text{Ad}(w^*_{n-1}\dots \allowbreak w^*_1)
(a_i) : i \le n}$, $\epsilon_n$ to find a $G'_n \Subset A$ and $\delta'_n >0$
which satisfy the thesis of the lemma (with item \ref{item4b}).
\item[\emph{b2)}] \label{naim:b2}
Fix $\sigma = (u^*_{n-2}v_{n-1})^{k'} \sigma_{j'}$, $k=k'$,
$K = G'_n \cup F'_n$, $\epsilon' = \min \set{\delta'_n, 2^{-n}}$,
and let $(v_{n,t})_{
t \in [0,1]}$ be a path of unitaries in $A$ given by the application of lemma \ref{lemma:1+} in
part \emph{a1} such that (we will denote $v_{n,1}$ simply by $v_n$):
\begin{enumerate}[(i)]
\item $v_{n,0} = 1$,
\item $\phi_h \circ \text{Ad} (v_1 \dots v_n) \approx_{K, \epsilon'} \tilde{\psi}_h \circ \text{Ad} (w_1 \dots w_{n-1})$
for $h \le n$,
\item $\lVert b - \text{Ad}(v_{n,t})(b) \rVert < \epsilon_n$ for all $b \in F_n$.
\item $\lVert \text{Ad}((u^*_{n-2}v_{n-1})^{k'} \sigma_{j'}) ( b_{k',j'})
- \text{Ad} (v^{*k'}_n) (b_{k',j'}) \rVert_{2,\tau} > 1/4$ for all $\tau \in T(A)$.
\end{enumerate}

\end{itemize}
\end{description}
The proof that $\Phi$ and $\Psi$, defined respectively as the pointwise limits
of $\set{\text{Ad}(v_n)}_{n \in \N}$ and $\set{\text{Ad}(w_n)}_{n \in \N}$,
are two automorphisms of $A$ such that $\phi_h \circ \Phi \sim \psi_h \circ \Psi$
for all $h \in \N$ is as in theorem 2.1 of \cite{KOS}.
Suppose now that $\alpha = \Psi \circ \Phi^{-1}$, and that $\alpha^k$ is a
$\tau$-weakly inner automorphism for some
$k \in \N$ and $\tau \in T(A)$. There exists $\sigma_j$ such that, for all $a \in A^1$
\[
\lVert \text{Ad}(\sigma_j)(a) - \alpha^k(a) \rVert_{2,\tau} \le 1/16.
\]
Let $N \in \N$ be corresponding to the step $(j,k)$ in the induction.
Let $n \in \N$ be bigger than $N$ and such that
$\lVert \text{Ad}(u_n^k)(b_{j,k}) - \alpha^k(b_{j,k}) \rVert_{2,\tau} < 1/16 $.
Hence by construction it follows that
\[
\lVert \text{Ad} (\sigma_j)(b_{j,k}) - \text{Ad}(u_n^k)(b_{j,k}) \rVert_{2,\tau} > 1/8,
\]
which is a contradiction.
Let $B = A \rtimes_\alpha \mathbb{Z}$ and let $u_\alpha \in B$
be the unitary which with $A$ generates $B$ and such that
$u_\alpha a u_\alpha^* = \alpha(a)$ for all $a \in A$. The \cstar-algebra
$B$ is generated by $\set{a_i}_{i \in \N} \cup \set{u_\alpha}$, therefore, in order to
check that $B' \cap B^{\mathcal{U}}$ contains a copy of $M_2(\C)$, it is
enough to prove clause \ref{ni2} of proposition \ref{prop:comm} for finite subsets of
$\set{a_i}_{i \in \N} \cup \set{u_\alpha}$. Let $\epsilon > 0$ and $F \Subset \set{a_i}_{i \in \N}
\cup \set{u_\alpha}$ containing $u_\alpha$. Let $n \in \N$ be big enough so that
$F \setminus \set{u_\alpha} \subseteq F_n$ and $\epsilon_{n-4} < \epsilon$. We have that
\begin{enumerate}
\item $\max \set{ \lVert c_n \lVert, \rVert d_n \rVert} \le 1$,
\item $\max \set{\lVert c_nc_n^*-1 \rVert_{2,u}, \lVert d_nd_n^* - 1
\rVert,\lVert c_n^*c_n-1 \rVert_{2,u},
\lVert d_n^*d_n - 1 \rVert_{2,u}, \lVert c_nd_n + d_nc_n \rVert_{2,u}} < \epsilon_{n-4} < \epsilon$,
\item $\max_{b \in F \setminus \set{u_\alpha}}
\set{\lVert [c_n,b] \rVert_{2,u} , \lVert [d_n,b] \rVert_{2,u}} < \epsilon_{n-4} < \epsilon$.
\end{enumerate}
Moreover, by part \emph{b1} of each step $m \ge n$ of the construction, we infer that,
for every $m \ge n$ and $x_n \in \set{c_n, d_n}$
\[
\lVert \text{Ad}(u_m)(x_n) - x_n \rVert_{2,u} < \epsilon_{n-4} < \epsilon.
\]
Therefore, we conclude that
\[
\lVert u_\alpha x_n u_\alpha^* - x_n \rVert_{2, \tau} = \lVert \alpha(x_n) - x_n \rVert_{2, \tau} =
\lim_{m \to \infty} \lVert \text{Ad}(u_m)(x_n) - x_n \rVert_{2,\tau} < \epsilon.
\]
\end{proof}

\begin{proof}[Proof of theorem \ref{thrm:a} - part \ref{ta:i2}]
Fix a dense $\set{a_i}_{i \in \N}$ in $A$.
\begin{description}
\item[Step 1.]
\begin{itemize}
\item[\emph{a1)}] Apply lemma \ref{lemma:2+} to $\phi_1$ for $F_1 = \set{a_1}$,
$\epsilon_1 = 2^{-1}$, $\set{\tau_1}$, to find a finite $G_1 \subset A$ and
$\delta_1 >0$
which satisfy the thesis of the lemma (with item \ref{item4a}).
\item[\emph{b1)}] Fix $\tilde{\psi}_1 \sim \psi_1$ such that $\tilde{\psi}_1 \approx_{G, \delta} \phi_1$.
\item[\emph{a2)}] Apply lemma \ref{lemma:2+} to $\tilde{\psi}_1$ for $F'_1=F_1$,
$\epsilon_1$,
$\set{\tau_1}$, to find a finite $G'_1 \subset A$ and $\delta'_1> 0$
which satisfy the thesis of the lemma (with item \ref{item4a}).
\item[\emph{b2)}] Fix $K = G'_1 \cup F'_1$ and $\epsilon' = \min \set{\delta'_1, 1/2}$, and let $(v_{1,t})_{
t \in [0,1]}$ be a path of unitaries in $A$ given by the application of lemma \ref{lemma:2+} in
part \emph{a1} such that (we will denote $v_{1,1}$ simply by $v_1$):
\begin{itemize}
\item $v_{1,0} = 1$,
\item $\phi_1 \circ \text{Ad} (v_1) \approx_{K, \epsilon'} \tilde{\psi}_1$,
\item $\lVert b - \text{Ad}(v_{1,t})(b) \rVert < \epsilon_1$ for all $b \in F_1$,
\item $\lVert v_1 - 1 \rVert_{2,1} < \epsilon_1$.
\end{itemize}
\end{itemize}

\item[Step n.]
\begin{itemize}
\item[\emph{a1)}] Apply lemma \ref{lemma:2+} to $(\phi_h \circ \text{Ad} (v_1 \dots
v_{n-1}))_{h \le n}$ for $F_n = F'_{n-1} \cup
\set{a_i, \allowbreak \text{Ad}(v^*_{n-1}\dots v^*_1)
(a_i) : i \le n}$, $\epsilon_n = 2^{-n}$, $\set{\tau_1, \dots, \tau_n}$
to find a finite $G_n \subset A$ and $\delta_n > 0$
which satisfy the thesis of the lemma (with item \ref{item4a}).
\item[\emph{b1)}] Fix $K = G_n \cup F_n$ and $\epsilon' = \min \set{\delta_n, 2^{-n}}$,
and let $(w_{n-1,t})_{
t \in [0,1]}$ be a path of unitaries in $A$ given by the application of lemma \ref{lemma:2+} in
part \emph{a2} of the previous step such that (we will denote $w_{n-1,1}$ simply by $w_{n-1}$):
\begin{itemize}
\item $w_{n-1,0} = 1$,
\item $\phi_h \circ \text{Ad} (v_1 \dots v_{n-1}) \approx_{K, \epsilon'}
\tilde{\psi}_h \circ \text{Ad} (w_1 \dots w_{n-1})$ for $h \le n-1$,
\item $\lVert b - \text{Ad}(w_{n-1})(b) \rVert < \epsilon_{n-1}$ for all $b \in F'_{n-1}$,
\item $\lVert w_{n-1} - 1 \rVert_{2,k} < \epsilon_{n-1}$ for all $k \le n-1$.
\end{itemize}
\item[\emph{a2)}] Apply lemma \ref{lemma:2+} to
$(\tilde{\psi}_h \circ \text{Ad} (w_1 \dots w_{n-1}))_{h \le n}$ for
$F'_n = F_n \cup \set{\text{Ad}(w^*_{n-1}\dots \allowbreak w^*_1)
(a_i) : i \le n}$, $\epsilon_n$, $\set{\tau_1, \dots, \tau_n}$
to find a finite $G'_n \subset A$ and $\delta'_n >0$
which satisfy the thesis of the lemma (with item \ref{item4a}).
\item[\emph{b2)}] Fix $K = G'_n \cup F'_n$ and $\epsilon' = \min \set{\delta'_n, 2^{-n}}$,
and let $(v_{n,t})_{
t \in [0,1]}$ be a path of unitaries in $A$ given by the application of lemma \ref{lemma:2+} in
part \emph{a1} such that (we will denote $v_{n,1}$ simply by $v_n$):
\begin{itemize}
\item $v_{n,0} = 1$,
\item $\phi_h \circ \text{Ad} (v_1 \dots v_n) \approx_{K, \epsilon'} \tilde{\psi}_h \circ \text{Ad} (w_1 \dots w_{n-1})$
for $h \le n$,
\item $\lVert b - \text{Ad}(v_{n,t})(b) \rVert < \epsilon_n$ for all $b \in F_n$,
\item $\lVert v_n - 1 \rVert_{2,k} < \epsilon_n$ for all $k \le n$.
\end{itemize}

\end{itemize}

\end{description}
The proof that the maps $\Phi$ and $\Psi$, defined respectively as the pointwise limits
of $\set{\text{Ad}(v_n)}_{n \in \N}$ and $\set{\text{Ad}(w_n)}_{n \in \N}$,
are two automorphisms of $A$ such that $\phi_h \circ \Phi \sim \psi_h \circ \Psi$
for all $h \in \N$ is as in theorem 2.1 in \cite{KOS}.
If we let $u_t = w_t v^*_t$, then the path of unitaries $(u_t )_{t \in [0, \infty)}$ is such
that $\alpha(a) = \lim_{t \to \infty} \text{Ad} (u_t) (a) $ for all $a \in A$ is the required automorphism.
By construction, for each $n \in \N$ and all $k \le n$ we have that
\[
\lVert u_{n+1} - u_n \rVert_{2,k} = \lVert u_{n+1} u^*_n -1 \rVert _{2,k} =
\lVert w_{n+1} v^*_{n+1} -1 \rVert_{2,k} < 2^{-(n -1)}
\]
Thus, given any $\tau \in \set{\tau_k}_{k \in \N}$,
the sequence $\set{\pi_\tau(u_n)}_{n \in \N}$ is strongly convergent on $B(H_\tau)$
(recall that the strong convergence of $\set{\pi_\tau(u_n)}_{n\in \N}$ is equivalent
to the convergence of $\set{u_n}_{n\in \N}$ in the $\ell_2$-norm induced by $\tau$).
Let $v$ be its strong limit. Then $ \text{Ad}(v)$ extends $\alpha$, in fact for every
$a,x,y \in A$ and $\epsilon > 0$, for $n \in \N$ big enough the following holds
\[
\langle v \pi_\tau(a) v^*x, y \rangle_\tau = \langle \pi_\tau(a)v^*x,v^*y \rangle_\tau
\approx_\epsilon \langle \pi_\tau(a u_n^*)x, \pi_\tau(u_n^*)y \rangle_\tau =
\]
\[
= \langle \pi_\tau(u_n a u_n^*) x, y \rangle_\tau
\approx_\epsilon \langle \pi_\tau(\alpha(a))x, y \rangle_\tau
\]
The argument extends by density to all $x,y \in H_\tau$ and all $a \in \pi_\tau(A)''$.

\end{proof}

\section{Conclusions and final remarks} \label{sctn:5}
Going back to the main motivation of our inquiry, namely understanding what counterexamples
to Naimark's problem look like and how they could be characterized,
we are still not able to say anything more
that such algebras have to be nonseparable, simple and non-type I. The results we proved
actually show that the tracial simplex of a counterexample
to Naimark's problem has no specific property, at least when it is separable.
On the other hand, theorem \ref{thrm:1} provides a wide variety of counterexamples, and it
highlights the versatility of the techniques in \cite{KOS} and \cite{AW}. It would be interesting
to know how further this versatility can be pushed, to see for instance
if it is possible to obtain any (nonseparable) Choquet simplex
as the trace space of a counterexample to Naimark's problem, or
if there is any K-theoretic or model theoretic obstruction to being a counterexample
to Naimark's problem.

Another interesting topic (already mentioned in the introduction of \cite{farahilan})
is the existence of a counterexample to Naimark's problem
with an outer automorphism. This problem is related to the following theorem.
\begin{theorem}[{\cite[Theorem 2.1]{outer}}] \label{thrm:out}
Let $A$ be a separable simple \cstar-algebra and $\alpha \in \text{Out}(A)$.
Then there exist two inequivalent pure states $\phi, \psi \in P(A)$ such that
$\phi = \psi \circ \alpha$
\end{theorem}
This result
is linked in turn to the following question on inner automorphisms which, to our
knowledge, is open.
\begin{question}
Let $A$ be a unital \cstar-algebra and let $\alpha$ be an automorphism of $A$. Suppose
that, whenever $A$ is embedded in a \cstar-algebra $B$, $\alpha$ extends to an
automorphism of $B$. Is $\alpha$ inner?
\end{question}
The analogous question has a positive answer for the category of groups (see \cite{schupp}),
and an application of theorem \ref{thrm:out} shows that this is also the case for separable simple unital \cstar-algebras.
In fact, let $A$ be a separable simple unital \cstar-algebra and $\alpha \in \text{Out}(A)$.
Suppose $\phi, \psi \in P(A)$ are two inequivalent pure states such that
$\phi = \psi \circ \alpha$. Since $A$ is simple, the GNS representation associated
to $\phi$ provides a map $\pi_\phi: A \to B(H_\phi)$
which is an embedding of $A$ into $B(H_\phi)$.
Identify $A$ with $\pi_\phi(A)$ and suppose
$\alpha$ can be extended to an automorphism of $B(H_\phi)$,
which means that there is $u \in \mathcal{U}(B(H_\phi))$
such that $\text{Ad}(u)\restriction_A = \alpha$. The pure state $\psi$ is thus equal to the
vector state induced by $u\xi_\phi$, therefore an application of the Kadison transitivity theorem
entails that $\phi$ and $\psi$ are unitarily equivalent, which is a contradiction.
A generalization of theorem \ref{thrm:out}
to nonseparable \cstar-algebras would settle
the question also in the nonseparable simple case. However,
a counterexample to Naimark's problem with an outer
automorphism would witness the impossibility of such generalization.

\begin{acknow}
I am grateful to Ilijas Farah for having introduced me to this problem, and
for the countless suggestions he gave me during the early stages of this work
and on the first drafts of this paper. I wish to thank George Elliott for sharing with
me interesting remarks and comments on Naimark's problem, and Georgios
Katsimpas for his useful feedback on the earlier drafts of this paper.
I further thank Alessandro Vignati for raising the question on the tracial simplexes of
counterexamples to Naimark's problem in the first place. I am grateful to the anonymous
referee for their useful suggestions and remarks on the paper, and in particular on the abstract.
I wish to thank G{\'a}bor Szab\'o for pointing out the mistake in the original version of the paper.
Finally, I wish to thank Christopher Schafhauser who, along with G{\'a}bor,
made crucial suggestions in the process of fixing it.
\end{acknow}

\bibliographystyle{amsalpha}
	\bibliography{Bibliography}

\end{document}